\documentclass[12pt]{amsart}

\usepackage{amsmath,amsthm,amscd,amsfonts,amssymb,epic,eepic,bbm,graphicx}
\usepackage[pagebackref,colorlinks=true,linkcolor=blue,citecolor=blue]{hyperref}

\allowdisplaybreaks
\newtheorem{thm}{Theorem}[section]

\newtheorem{lem}[thm]{Lemma}

\newtheorem{prop}[thm]{Proposition}
\theoremstyle{remark}

\newcounter{remarkscounter}

\numberwithin{equation}{section}
\newcommand{\A}{\mathbb{A}}
\newcommand{\GL}{\mathrm{GL}}
\newcommand{\gl}{\mathfrak{gl}}
\newcommand{\SL}{\mathrm{SL}}
\newcommand{\ZZ}{\mathbb{Z}}

\newcommand{\QQ}{\mathbb{Q}}

\newcommand{\lto}{\longrightarrow}

\newcommand{\OO}{\mathcal{O}}
\newcommand{\CC}{\mathbb{C}}
\newcommand{\RR}{\mathbb{R}}
\newcommand{\GG}{\mathbb{G}}

\newcommand{\quash}[1]{}

\theoremstyle{definition}
\newtheorem{defn}[thm]{Definition}

\renewcommand{\bar}{\overline}

\numberwithin{equation}{subsection}

\newcommand{\one}{\mathbbm{1}}

\begin{document}

\title{A refined Poisson summation formula for certain Braverman-Kazhdan spaces}

\author{Jayce R. Getz}
\address{Department of Mathematics\\
Duke University\\
Durham, NC 27708}
\email{jgetz@math.duke.edu}

\author{Baiying Liu}
\address{Department of Mathematics\\
Purdue University\\
West Lafayette, IN 47907}
\email{liu2053@purdue.edu}


\thanks{The first named author is thankful for partial support provided by NSF grant DMS 1405708. The second named author is partially supported by NSF grant DMS 1702218 and by a start-up fund from the Department of Mathematics at Purdue University.
Any opinions, findings, and conclusions or recommendations expressed in this material are those of the authors and do not necessarily reflect the views of the National Science Foundation.}

\begin{abstract}
Braverman and Kazhdan \cite{BK-lifting} introduced influential conjectures generalizing the Fourier transform and the Poisson summation formula.   Their conjectures should imply that quite general Langlands $L$-functions have meromorphic continuations and functional equations as predicted by Langlands' functoriality conjecture.  As evidence for their conjectures, Braverman and Kazhdan considered a setting related to the so-called doubling method in a later paper \cite{BK:normalized} and proved the corresponding Poisson summation formula under restrictive assumptions on the functions involved.  The connection between the two papers is made explicit in \cite{WWLi:Zeta}.  In this paper we consider a special case of the setting of \cite{BK:normalized}, and prove a refined Poisson summation formula that eliminates the restrictive assumptions of loc.~cit.  Along the way we provide analytic control on the Schwartz space we construct; this analytic control was conjectured to hold (in a slightly different setting) in \cite{BK:normalized}.
\end{abstract}

\maketitle

\tableofcontents

\section{Introduction}

Let $F$ be a number field and $\A_F$ its ring of adeles.  Let $f \in \mathcal{S}(\gl_n(\A_F))$ be a Schwartz function.  Then the Poisson summation formula on $\gl_n(F)$ asserts that
$$
\sum_{\gamma \in \gl_n(F)} f(\gamma)=\sum_{\gamma \in \gl_n(F)} \widehat{f}(\gamma)\,,
$$
where $\widehat{f}$ is the Fourier transform of $f$.  Following Tate, who considered the case $n=1$ in his thesis, Godement and Jacquet \cite{GodementJacquetBook} used this formula to prove that the standard $L$-function of a cuspidal automorphic representation of $\GL_n(\A_F)$ has a holomorphic continuation to the plane and a functional equation.  

Braverman and Kazhdan \cite{BK-lifting} have suggested that this is but the first case of a general phenomenon. They conjecture that for every split reductive group $G$ and representation $\rho:{}^LG^{\circ} \lto \GL(V_\rho)$
of the neutral component of the $L$-group there is a corresponding nonabelian Poisson summation formula.  The summation formula should imply the functional equation and meromorphic continuation of the Langlands $L$-functions $L(s,\pi,\rho)$ for $\pi$ a cuspidal automorphic representation of $G(\A_F)$.  The replacement for $\gl_n(F)$ is a certain reductive    
 monoid attached to $\rho$ using results of Vinberg \cite{NgoSums} that can be viewed as a sort of compactification of $G$.  These are referred to as ``nonabelian'' Poisson summation formulae because in general there is no additive structure on the reductive monoid, only the multiplicative 
structure extending the group multiplication on $G$.

\subsection{The Poisson summation formula for Braverman-Kazhdan spaces}

Let $G$ be a split reductive group with simply connected derived group over a number field $F$ and let $P \leq G$ be a proper parabolic subgroup.  Let 
$$
X:=[P,P] \backslash G\,,
$$
where $[H,H]$ denotes the derived group of an algebraic group $H$.
Braverman and Kazhdan \cite{BK:normalized} defined a space of Schwartz functions on $X(\A_F)$ and 
sketched a proof of a Poisson summation formula for this space of functions.  At least in certain cases, the spaces $X$ can be related to reductive monoids attached to the standard representations of (the $L$-groups of) classical groups, and the Poisson summation formula of Braverman and Kazhdan provides a different perspective from which one can view the famous doubling method introduced by  Rallis and Piatetski-Shapiro \cite{GPSR:LNM}.  Thus Braverman and Kazhdan were able to confirm their conjectures on nonabelian Poisson summation formulae in this case.

The argument ultimately boils down to an application of the theory of Eisenstein series, just as in the doubling method. Braverman and Kazhdan's real achievement was finding a geometric way to interpret and normalize the intertwining operators that are used to study the meromorphic continuation of these Eisenstein series.  In honor of their work we refer to $X$ as a Braverman-Kazhdan space.  

Unfortunately, the description of the Schwartz space given by Braverman and Kazhdan makes the growth properties of functions in the space unclear.     Moreover, they imposed conditions to eliminate boundary terms in the Poisson summation formula that are vital in applications.  For example, if one were using their formula to reprove the analytic continuation and functional equation of triple product $L$-functions then one would not be able to say anything about the residues of these $L$-functions (see  \cite{PSRallisTriple,Ikeda:poles:triple}).  

In this paper we explicate and refine Braverman and Kazhdan's work in a special case.  
Let 
$$
J:=\left(\begin{smallmatrix} & I_n \\ -I_n &  \end{smallmatrix} \right),
$$
and for $\ZZ$-algebras $R$ let
$$
\mathrm{Sp}_{2n}(R):=\left\{g \in \GL_{2n}(R): gJ^{-1}g^tJ=I_{2n}\right\}.
$$
 We typically regard $\mathrm{Sp}_{2n}$ as a reductive group over $F$ or one of its completions.  For $F$-algebras $R$ let 
\begin{align} \label{Siegel}
P(R):=\{\left(\begin{smallmatrix} A &  \\ &A^{-t} \end{smallmatrix}\right)\left(\begin{smallmatrix} I_n &  Z \\ &I_n \end{smallmatrix}\right): A \in \GL_n(R): Z=Z^t\}\,.
\end{align}
We let $M \leq P$ be the Levi subgroup of block diagonal matrices and let $N \leq P$ be the unipotent radical.

Let $K \leq \mathrm{Sp}_{2n}(\A_F)$ be a maximal compact subgroup such that $K^\infty$ is $\mathrm{Sp}_{2n}(\A_F^\infty)$-conjugate to $\mathrm{Sp}_{2n}(\widehat{\OO})$.  Here $\OO$ is the ring of integers of $F$.  We define a Schwartz space $\mathcal{S}(X(\A_F),K)$ and construct a Fourier transform
$$
\mathcal{F}:=\mathcal{F}_{\psi,K}:\mathcal{S}(X(\A_F),K) \lto \mathcal{S}(X(\A_F),K)
$$
depending on a (nontrivial) additive character $\psi:F \backslash \A_F \to \CC^\times$ (see Theorem \ref{thm:FT}).  Here the $K$ indicates that the functions in the space are $K$-finite.  We also develop the analytic properties of elements in the Schwartz space, including growth estimates (see \S \ref{sec:Schwartz}).

We then obtain a Poisson summation formula:
\begin{thm} \label{thm:intro}
Let $\Phi \in \mathcal{S}(X(\A_F),K)$.  One has that
\begin{align*}
&\sum_{\gamma \in X(F)}\Phi(\gamma)+
\frac{1}{\kappa_F}\sum_{
\substack{
0 \leq m <\frac{n+1}{2} \\m \in \mathbb{Z}}}\mathrm{Res}_{s= \frac{n+1}{2}-m }E(\mathcal{F}(\Phi)_{1_s})+\frac{1}{\kappa_F}\sum_{\substack{\chi \in \widehat{[\GG_m]} \\\chi \neq 1, \chi^2=1}}
\sum_{\substack{
0 \leq m <\frac{n-1}{2} \\m \in \mathbb{Z}}}
\mathrm{Res}_{s=\frac{n-1}{2}-m }E(\mathcal{F}(\Phi)_{\chi_s})
\\
&=\sum_{\gamma \in X(F)} \mathcal{F}(\Phi)(\gamma)+
\frac{1}{\kappa_F}\sum_{
\substack{
0 \leq m <\frac{n+1}{2} \\m \in \mathbb{Z}}}\mathrm{Res}_{s= \frac{n+1}{2}-m }E(\Phi_{1_s})+\frac{1}{\kappa_F}\sum_{\substack{\chi \in \widehat{[\GG_m]} \\\chi \neq 1, \chi^2=1}}
\sum_{\substack{
0 \leq m <\frac{n-1}{2} \\m \in \mathbb{Z}}}
\mathrm{Res}_{s=\frac{n-1}{2}-m }E(\Phi_{\chi_s})\,.
\end{align*}
All of the sums here are absolutely convergent. 
\end{thm}

Here $E(\Phi_{\chi_s}):=E(I_{2n},\Phi_{\chi_s})$ is a certain degenerate Siegel Eisenstein series (see \eqref{Eis0}).  

\subsection{Motivation}  
We make explicit three motivations for proving Theorem \ref{thm:intro} and the refined
definition of the Schwartz space that underlies it.  The first motivation is that 
these sorts of Poisson summation formulae are interesting for their own sake.  Given the utility of the usual Poisson summation formula on vector spaces, one expects that Poisson summation formulae on any more general schemes will be extremely useful.  For example, one could use the formula to give sharp estimates for the number of points of $X(F)$ in suitable sets (see \cite{Franke:Manin:Tschinkel,DukeRudnickSarnak}
 for counting results of a similar flavor).
 
The second motivation is in line with Braverman and Kazhdan's original motivation.  Nonabelian Poisson summation formulae are expected to imply the functional equations and meromorphic continuation of Langlands $L$-functions.  Thus
 developing notions of Fourier transforms and Schwartz spaces that are analytically tractable in more general contexts is an extremely important problem.  We hope that our work here will shed light on the (conjectural) general picture.  There has been other work on this important question, and we mention in particular the work in 
\cite{BouthierNgoSakellaridis,Cheng:Ngo:BK,GetzBK,
LafforgueJJM,WWLiSat,WWLi:Zeta,WWLi:Towards,SakellaridisSat,SakellaridisSph,Shahidi:LF}.

Our final motivation is that we plan to use the formula in Theorem \ref{thm:main} to prove an entirely new Poisson summation formula for a certain homogeneous space.  We sketch this application.  The details will be given in a sequel to this paper.  Let $(V_i,Q_i)$, $1 \leq i \leq 3$ be a triple of vector spaces of even dimension equipped with nondegenerate quadratic forms $Q_i$.  We will prove a Poisson summation formula for the affine $F$-scheme that is the zero locus of $Q_1-Q_2$ and $Q_2-Q_3$ on $V_1 \oplus V_2 \oplus V_3$ (in other words, the triples in $V_1 \oplus V_2 \oplus V_3$ on which the $Q_i$ have the same value).  The proof will involve integrating the Poisson summation identity of Theorem \ref{thm:intro} in the case $n=3$ against a product of the three $\theta$-functions attached to the quadratic forms $Q_i$.  In other words, we will apply Garrett's triple product $L$-function construction to three $\theta$-functions.  We hope to then apply this to study triple product $L$-functions of higher rank groups.

\subsection{Sketch of the proof}

The idea of the proof of Theorem \ref{thm:intro} is due to Braverman and Kazhdan and the formal argument is straightforward.  Let $P=MN$ where $N$ is the unipotent radical and $M$ is the Levi subgroup of block-diagonal matrices.  Moreover let $M^{\mathrm{ab}}:=M/[M,M]$.
For a reductive group $G$, we let 
$$
[G]:=A_G G(F) \backslash G(\A_F)\,,
$$
where $A_G$ is the neutral component in the real topology of the $\RR$-points of the maximal $\QQ$-split torus in $\mathrm{Res}_{F/\QQ}Z_G$.  

For Hecke characters $\chi \in \widehat{[\GG_m]}$ and $s \in \CC$ let
$$
\chi_s:=\chi|\cdot|^s\,,
$$
where $|\cdot|:[\GG_m] \to \CC^\times$ is the typical idelic norm.  
For each $\Phi \in \mathcal{S}(X(\A_F),K)$ and $s \in \CC$ with sufficiently large real part we define a Mellin transform
\begin{align*}
\Phi_{\chi_{s}}(g)=\int_{ M^{\mathrm{ab}}(\A_F)}\delta_P(m)^{1/2}\chi_{s}\left( \omega( m) \right)\Phi(m^{-1}g)dm\,.
\end{align*}
Then
\begin{align} \label{Eis0}
E(g,\Phi_{\chi_s}):=\sum_{\gamma \in P(F) \backslash \mathrm{Sp}_{2n}(F)}\Phi_{\chi_s}(\gamma g)
\end{align}
is a degenerate Siegel Eisenstein series, induced from the twist of the trivial representation on $M$ by $\chi_s$.  In fact, by Mellin inversion,  for $\sigma \in \RR$ sufficiently large one has
$$
\sum_{\gamma \in X(F)} \Phi(\gamma)=\sum_{\chi \in \widehat{[\GG_m]}} \frac{1}{2 \pi i \kappa_F} \int_{\mathrm{Re}(s)=\sigma}E(I_{2n},\Phi_{\chi_s})\,,
$$
where $\kappa_F:=\mathrm{Res}_{s=1} \zeta_F(s)$.  We now apply Langlands' functional equation for $E(I_{2n},\Phi_{\chi_s})$ to replace this by 
$$
\sum_{\chi \in \widehat{[\GG_m]}} \frac{1}{2 \pi i \kappa_F} \int_{\mathrm{Re}(s)=\sigma}E(I_{2n},M_{w_0}^*(\Phi_{\chi_s}))\,,
$$
where $M_{w_0}^*$ is Langlands' intertwining operator attached to the long Weyl element $w_0$.  The $*$ in the superscript indicates that it is normalized as in \cite{Ikeda:poles:triple}.  Now this Eisenstein series converges absolutely for $\mathrm{Re}(s)$ very negative.  Thus we shift the contour to $-\sigma$, picking up the poles of the Eisenstein series along the way.  Finally we apply Mellin inversion again.  Our Fourier transform $\mathcal{F}(\Phi)$ is designed so that 
$$
\mathcal{F}(\Phi)_{\chi_s}=M_{w_0}^* (\Phi_{\bar{\chi}_{-s}})
$$
(see \eqref{FT} and the diagram directly following it).  This allows us to deduce the main theorem.  For the complete argument we refer the reader to \S \ref{sec:global:sum}.  

This outline hides the substantial subtleties involved in making this argument rigorous and the related problem of establishing analytic control of the Schwartz space.  The lion's share of this is omitted or stated as conjectures in \cite{BK:normalized}.  This is not meant as a criticism of Braverman and Kazhdan's work.  It is only meant to explain why the current paper is a necessary addition to the literature.  We will mention some of the subtleties in the following section when we outline the contents of the paper.

\subsection{Outline of the paper}

In \S \ref{sec:BKspace} we explain how the geometry of $X$ can be understood using a lift of the Pl\"ucker embedding of $P \backslash \mathrm{Sp}_{2n}$ into an appropriate projective space.  This Pl\"ucker embedding 
allows us to define a function $|x|$ on $x \in X(F)$ ($F$ a local field) which measures its size.  This plays a role in describing the asymptotic behavior of functions on $X$.  

The analytic properties of Eisenstein series that we require are proven in \cite{Ikeda:poles:triple}.  We review Ikeda's construction of ``good sections'' and refine it in \S \ref{sec:normalized}.  This is then used in \S \ref{sec:Schwartz} to define the local Schwartz space and the Fourier transform.  In the non-Archimedean case we use work of Ikeda to show that compactly supported smooth functions on $X(F)$ ($F$ a local field) are in the  Schwartz space.  This is also true in the Archimedean case, but we defer the proof to Appendix \ref{App}.

We then give analytic control on the Schwartz space in \S \ref{sec:control}.  In particular, we prove that functions in the Schwartz space have polynomial growth as $|g| \to 0$ and are rapidly decreasing as $|g| \to \infty$. Finally, in \S \ref{sec:global:sum}, we prove Theorem \ref{thm:intro}, which is restated in that section as Theorem \ref{thm:main}.

\subsection{Notation and measures} \label{sec:notat}

In this paper $F$ refers to a number field or a completion of it.  When $F$ is a number field or non-Archimedean local field we denote by $\OO$ its ring of integers.  If $F$ is a local non-Archimedean field we let $\varpi$ be a uniformizer, set $q:=|\OO/\varpi|$, and let $|\cdot|$ be the usual norm, so $|\varpi|=q^{-1}$.  If $F$ is an Archimedean local field $|\cdot|$ is the standard norm if $F=\RR$ and the square of the standard norm if $F=\CC$.  This may cause some confusion when we deal with $\CC$-valued functions so we set 
$$
|z|_{\mathrm{st}}:=(z\bar{z})^{1/2}
$$
(the positive square root) for $z \in \CC$.  

In our derivation of Theorem \ref{thm:intro} we apply Mellin inversion.  Our main reference for this is \cite{Blomer_Brumley_Ramanujan_Annals}, so we use their measure conventions.  In more detail we normalize the Haar measure on 
the local field $F$ as follows:
\begin{center}
\begin{tabular}{l|l}
$F$ & $dx$\\
\hline
$\RR$ & Lebesgue measure \\
$\CC$ & twice Lebesgue measure\\
non-Archimedean & $dx(\OO)=|\mathfrak{d}|^{1/2}$
\end{tabular} 
\end{center}
Here in the non-Archimedean case $\mathfrak{d}$ is a generator of the absolute different.  To be more explicit, the Haar measure $dz$ on $\CC$ is $d(x+iy)=2dxdy$ where $dx$ and $dy$ are the usual Lebesgue measures on $\RR$.
We then let the Haar measure on $F^\times$ be
$$
dx^\times:=\zeta(1)\frac{dx}{|x|}\,.
$$
where $\zeta(s)$ is the Tate local zeta function of $F$.

We use the standard analytic number theory symbol
$$
A \ll_B C
$$
to mean that there is a constant $\kappa \in \RR_{>0}$, possibly depending on $B$, such that $A<\kappa C$.  Moreover
$$
A \asymp_B C
$$
means $A \ll_B C$ and $C \ll_B A$.

\subsection{Acknowledgments}

The authors would like to thank H.~Jacquet and A.~Pollack for their interest in the results in this paper and for helpful comments and suggestions. The authors also thank W-W.~Li, Y.~Sakellaridis, and F.~Shahidi for useful conversations,  and thank F.~Shahidi and W-W.~Li for sharing \cite{Shahidi:FT,WWLicomparison} with us. The authors also would like to thank H.~Hahn for the help with editing and for her constant encouragement.

\section{Preliminaries on the space $X$} \label{sec:BKspace}

Recall that $X:=[P,P] \backslash \mathrm{Sp}_{2n}$ where $P \leq \mathrm{Sp}_{2n}$ is the Siegel parabolic of \eqref{Siegel} and $M \leq P$ is the Levi subgroup consisting of block diagonal matrices.
Let $M^{\mathrm{ab}}:=M/[M,M]$ be the abelianization of the Levi subgroup $M$.

\begin{lem} The natural maps 
\begin{align*}
\mathrm{Sp}_{2n}(F) &\lto X(F)\\
\mathrm{Sp}_{2n}(F) &\lto P \backslash \mathrm{Sp}_{2n}(F)\\
M(F) &\lto M^{\mathrm{ab}}(F)
\end{align*}
are all surjective.
\end{lem}  

\begin{proof}
To prove the first assertion it suffices to verify that $H^1(F,[P,P])=1$.  One has an exact sequence
$$
 H^1(F,N) \lto H^1(F,[P,P]) \lto H^1(F,\SL_n)\,,
$$
where the first map is induced by the inclusion of the unipotent radical $N$ of $P$ into $[P,P]$ and the second is induced by the quotient map to the maximal reductive quotient.  The left group is trivial \cite[\S III.2.1, Proposition 6]{Serre:GC} and the right group is trivial by \cite[\S III.3.2(a)]{Serre:GC}.

The last two assertions follow from similar arguments.
\end{proof}

\subsection{A Pl\"ucker embedding of $X$} \label{ssec:Plucker}
We can use the Pl\"ucker embedding to give a linear description of $X$.  More geometric information about this embedding can be found in \cite[\S 7.2]{WWLi:Zeta}.
We construct a commutative diagram
\begin{align} \label{Pl:diag}
\begin{CD}
X @>{\mathrm{Pl}}>> \wedge^n \GG_a^{2n} -\{0\} \\ @VVV  @VVV\\
P \backslash \mathrm{Sp}_{2n} @>>> \mathbb{P}(\wedge^n \GG_a^{2n})
\end{CD}
\end{align}
of morphisms of $F$-schemes as follows.  
The vertical arrows are the quotient maps.  
Let 
$e_i \in F^{2n}$ be the standard basis vector (with a $1$ in the $i$-th place and zeros elsewhere).  
Then $P$ is the stabilizer of the Lagrangian (a.k.a.~maximal isotropic) subspace 
$$
\langle e_{n+1},\dots,e_{2n} \rangle\,.
$$
The top arrow $\mathrm{Pl}$ sends $g$ to
$e_{n+1}g \wedge  \cdots \wedge e_{2n}g$, and the bottom arrow sends $g$ to the line spanned by this vector (this is just the usual Pl\"ucker embedding). 
In terms of matrices, for an $F$-algebra $R$ and $g=\left(\begin{smallmatrix} A\\ B \end{smallmatrix}\right) \in \mathrm{Sp}_{2n}(R)$ with $A,B \in M_{n \times 2n}(R)$  we have
\begin{align} \label{Pl}
\mathrm{Pl}(g)=b_1 \wedge \dots \wedge b_n\,,
\end{align}
where $b_i$ is the $i$-th row of $B$.

 If we let $\mathrm{Sp}_{2n}$ act on $X$ and $\GG_a^{2n}$ on the right then the horizontal arrows in \eqref{Pl:diag} are $\mathrm{Sp}_{2n}$-equivariant.  Let $R$ be an $F$-algebra.  There is a left action
\begin{align*}
M/M^{\mathrm{der}}(R) \times X(R) &\lto X(R)\\
(m,x) &\longmapsto mx\,.
\end{align*}
Define a character 
\begin{align} \label{omega0}\begin{split}
\omega:M(R) &\lto R^\times\\
\left(\begin{smallmatrix}  m & \\ & m^{-t} \end{smallmatrix}\right) &\longmapsto \det m\,. \end{split}
\end{align}
Thus $\omega$ induces a character of $M^{\mathrm{ab}}$.  By extending trivially on $N$ it also induces a character of $P$; we will use the symbol $\omega$ to denote all of these characters.
Then
\begin{align} \label{intertwine}
\mathrm{Pl}(mx)=\omega^{-1}(m)\mathrm{Pl}(x)\,.
\end{align}

\begin{lem}  \label{lem:inj} The map 
$$
\mathrm{Pl}:X(F) \lto \wedge^n F^{2n} -\{0\}
$$
is injective.
\end{lem}

\begin{proof}
In view of \eqref{intertwine}, a fiber of the map $X(F) \to P \backslash \mathrm{Sp}_{2n}(F)$ is mapped bijectively onto a fiber of the map $\wedge^n F^{2n}-\{0\} \to \mathbb{P}(\wedge^n F^{2n})$.  Since the Pl\"ucker embedding is injective we deduce the lemma.
\end{proof}

For the remainder of this section let $F$ be a local field of characteristic zero.  Throughout this paper in the archimedian case we let $K \leq \mathrm{Sp}_{2n}(F)$ be a maximal compact subgroup  and in the nonarchimedian case we let $K$ be an $\mathrm{Sp}_{2n}(F)$-conjugate of 
$$
K_0:=\mathrm{Sp}_{2n}(\OO).
$$
Thus in either case the Iwasawa decomposition $\mathrm{Sp}_{2n}(F)=P(F)K$ holds.
When considering the analytic properties of functions on $X(F)$ it is useful to have a way to measure the size of a point on $X(F)/K$.  This is what the constructions below will afford us (see Proposition \ref{prop:XmodK}).

Let $\mathrm{Sp}_{2n}(F)$ act on $F^{2n}$ on the right.  One obtains an induced action on $\wedge^n F^{2n}$.  When $F$ is Archimedean choose a positive definite bilinear form $(\cdot,\cdot)$ on $\wedge^n F^{2n}$ that is invariant under the action of $K$ and set $|x|=(x,x)^{[F:\RR]/2}$.  In the non-Archimedean case let $e_1,\dots,e_{2n}$ be the standard basis of $F^{2n}$ and let 
$$
\{e_{\alpha_1,\dots,\alpha_n}:=e_{\alpha_1} \wedge \dots \wedge e_{\alpha_n}:1 \leq \alpha_1<...<\alpha_n\leq 2n\}
$$
be the natural induced basis of $\wedge^n F^{2n}$.  Then
set
$$
\left|\sum_{1 \leq \alpha_{1}<\dots<\alpha_n \leq 2n}x_{\alpha_1,\dots,\alpha_n}e_{\alpha_1,\dots,\alpha_n}\right|=\max_{1\leq \alpha_1<\dots<\alpha_n \leq 2n}|x_{\alpha_1,\dots,\alpha_n}|\,.
$$
We claim that $|\cdot|$ is invariant under the natural action of $\GL(\wedge^n\OO^{2n})$ on the left and right.  
  To check this it suffices to treat the case where $x \in \wedge^{n}F^{2n}-\{0\}$.  In this case $|x|=q^{-k}$ where $k$ is the largest integer such that $\varpi^{-k}x \in \wedge^n \OO^{2n}$.  It is clear from this latter characterization that $|x|$ is preserved under $\GL(\wedge^n\OO^{2n})$.  In particular it is invariant under the action of $\mathrm{Sp}_{2n}(\OO)$ induced from its natural right action on $\OO^{2n}$.
 We then set 
\begin{align} \label{norm:def}
 |g|:=|\mathrm{Pl}(g)|\,.
\end{align}

For any $c \in \ZZ$ let
\begin{align} \label{c:def}
c(x):=\left( \begin{smallmatrix} x^{-c} & & \\
& I_{n-1} & \\ & & x^{c} & \\ & & & I_{n-1} \end{smallmatrix} \right).
\end{align}
In this way we obtain an isomorphism $\ZZ \cong X_*(M/M^{\mathrm{der}})$; we often use this isomorphism to identify integers with cocharacters of $M/M^{\mathrm{der}}$.  
With respect to this basis 
$|1(t)| \to 0$ as $|t| \to 0$.
The Iwasawa decomposition implies that 
$$
X(F)=\coprod_{c \in \ZZ} [P,P](F)c(\varpi)K_0
$$
in the non-Archimedean case, and 
$$
X(F)=\coprod_{t \in \RR_{>0}}[P,P](F)1(t)K
$$
in the Archimedean case.

In the non-Archimedean case for $c \in \ZZ$ set
\begin{align} \label{1c}
\one_c:=\one_{[P,P](F)c(\varpi)K_0}\,.
\end{align}
Then the functions $\one_c$, $c \in \ZZ$ form a basis for
$$
C_c^\infty(X(F)/K_0)
$$
by the Iwasawa decomposition.

\begin{prop} \label{prop:XmodK} 
Suppose $K=K_0$ when $F$ is nonarchimedian.  There is a continuous injection
\begin{align*}
X(F)/K \lto \RR_{>0} \\
[P,P](F)gK \longmapsto |g|\,.
\end{align*}
\end{prop}

\begin{proof}
It is easy to see that the map is well-defined and continuous. We only need to check injectivity.  In the Archimedean case one has 
$$
|1(t)|=|t||e_{n+1} \wedge \cdots \wedge e_{2n}|\,.
$$ 
Thus the map is injective.
If $F$ is non-Archimedean  then  $|c(\varpi)|=q^{-c}$, and thus we deduce injectivity in this case as well.
\end{proof}

\section{Normalized intertwining operators and excellent sections}
\label{sec:normalized}
For characters $\chi:F^\times \to \CC^\times$ and $s \in \CC^\times$ let $\chi_s:=\chi|\cdot|^s$.    
Let
\begin{align} \label{norm:ind}
I(\chi_{s}):=\mathrm{Ind}_{P}^{\mathrm{Sp}_{2n}}(\chi_{s})
\end{align}
(normalized induction) be the space of functions 
$f:\mathrm{Sp}_{2n}(F) \to \CC$ such that 
$$
f(pg)=\chi_{s}\left(\omega(p)\right)\delta_P(p)^{1/2}f(g)\,,
$$
where $\delta_P:P(F) \to \CC^\times$ is the modular quasi-character and $\omega$ is as in \eqref{omega0}.

We now recall some definitions from \cite[\S 1.2]{Ikeda:poles:triple}.
Let $E:=\CC[q^{-s},q^{s}]$ if $F$ is non-Archimedean and let $E$ be the ring of entire functions on $\CC$ if $F$ is Archimedean.
Assume $\chi$ is unitary. 
A function 
\begin{align*}
f^{(\cdot)}:\mathrm{Sp}_{2n}(F) \times \CC &\lto \CC\\
(g,s) &\longmapsto f^{(s)}(g)
\end{align*} is a \textbf{holomorphic section of $I(\chi_{s})$} if
\begin{enumerate}
\item For each $s \in \CC$, $f^{(s)} \in I(\chi_{s})$ (as a function of $g$),
\item For each $g \in \mathrm{Sp}_{2n}(F)$, $f^{(s)}(g) \in E$ as a function of $s$, and
\item The function $f^{(s)}$ is right $K$-finite.
\end{enumerate}
A function $f^{(\cdot)}$ on $\mathrm{Sp}_{2n}(F) \times \CC$ that is meromorphic in the second factor of the argument is a \textbf{meromorphic section} if there is an $\alpha \in E$ such that $\alpha f^{(\cdot)}$ is a holomorphic section.  

Let $T \leq \mathrm{Sp}_{2n}$ be the maximal torus of diagonal matrices, let $W_{\mathrm{Sp}_{2n}}$ be the Weyl group of $T$ in $\mathrm{Sp}_{2n}$ and let $W_M$ denote the Weyl group of $T$ in $M$.  Let $\Phi_{\mathrm{Sp}_{2n}}$ be the set of roots of $\mathrm{Sp}_{2n}$ with respect to $T$. 
We let
 $\Omega_n$ be the complete set of representatives for $W_{\mathrm{Sp}_{2n}} / W_M$ obtained by choosing the unique element of minimal length in each coset as follows. For each subset $I=\{i_1, i_2, \ldots, i_k\}$ of $\{1,2,\ldots,n\}$
let 
$$
J := \{j_1, j_2 \ldots, j_{n-k}\} = \{1,2,\ldots, n\} - I\,,
$$
 where $i_1<\dots<i_k$ and $j_1 <\dots <j_{n-k}$.
Define an element $w_I$ of $W_{\mathrm{Sp}_{2n}}$ by 
\begin{align*}
\begin{matrix}
& t_1 \mapsto t_{j_1}, &\ldots, &t_{n-k} \mapsto t_{j_{n-k}}\,,\\
& t_{n-k+1} \mapsto t^{-1}_{i_k}, & \ldots,& t_n \mapsto t^{-1}_{i_1}\,,
\end{matrix}
\end{align*}
where 
\begin{align*}\left(\begin{smallmatrix}
t_1 &&&&&\\
&\ddots &&&&\\
&&t_n &&&\\
&&&t_1^{-1}&&\\
&&&&\ddots &\\
&&&&&t_n^{-1}
\end{smallmatrix} \right)\in T(F)\,.
\end{align*}
In particular $w_0:=w_{\{1,\dots,n\}}$ is the long Weyl element which is conjugation by 
$$
\left(\begin{smallmatrix} & & && & -1\\ & & && \reflectbox{$\ddots$} & \\ & & & -1 & &\\ & & 1 & & &\\ & \reflectbox{$\ddots$}& & &\\ 1 & & & & \end{smallmatrix}\right).
$$
\quash{The length of $w_I$ is given by 
\begin{align}\label{length}
\begin{split}
\ell(w_I) = & \ \#\{\alpha \in \Phi_{\mathrm{Sp}_{2n}} | \alpha > 0, w_I \alpha < 0\}\\
= & \ \sum_{r=1}^k (n+1-i_r).
\end{split}
\end{align}}

For each $w \in \Omega_n$ and quasi-character $\chi:T(F) \to \CC^\times$ let $(\chi_s)^w(t):=\chi_s(w^{-1}tw)$ and let
$$
M_w:I(\chi_{s}) \lto I((\chi_{s})^w)
$$ 
denote the usual intertwining operator.  For $F$-algebras $R$ let
$$
B_{2n}(R):=\left\{\left(\begin{smallmatrix}A & * \\0 & A^{-t} \end{smallmatrix}\right) \in \mathrm{Sp}_{2n}(R):
A \textrm{ is upper triangular} \right\}.
$$
This is a Borel subgroup of $\mathrm{Sp}_{2n}$.  We let $N_{2n} \leq B_{2n}$ be the unipotent radical and let $N_{2n}^{\mathrm{op}}$ be the unipotent radical of the opposite Borel.
We recall that 
\begin{align} \label{inter}
M_{w}:I(\chi_{s}) &\lto I((\chi_{s})^{w})\\
f^{(s)} &\longmapsto \int_{N_{2n}(F) \cap wN^{\mathrm{op}}_{2n}(F)w^{-1}}f^{(s)}(w^{-1}ng)dn\,. \nonumber
\end{align}
   Here the integral is only well-defined for $\mathrm{Re}(s)$ sufficiently large; in general one has to define it via analytic continuation (\cite[Chapter 4]{Shahidi:book} is a nice reference).
To make the definition of $M_w$ precise we must fix the measure $dn$.  We proceed as follows:  For each $\alpha \in \Phi_{\mathrm{Sp}_{2n}}$ let $N_{\alpha} \leq N$ be the corresponding root subgroup; it comes equipped with an isomorphism of topological groups $F \tilde{\to} N_\alpha(F)$.  We let $dn_{\alpha}$ be the Haar measure on $N_{\alpha}(F)$ given by transporting the Haar measure on $F$ that is self-dual with respect to $\psi$ to $N_{\alpha}(F)$ via this isomorphism.  Then we let $dn=\prod dn_\alpha$ be the product measure, where the product is over the roots occurring in $N_{2n} \cap wN^{\mathrm{op}}_{2n}w^{-1}$.

Ikeda, following Piatetski-Shapiro and Rallis \cite{GPSR:LNM, PSRallisTriple}, found a convenient normalization for these intertwining operators \cite[\S 1.2, p.~195]{Ikeda:poles:triple}.  We now recall it because it plays a role in what follows.  For $w = w_I$, $I= \{i_1, i_2, \ldots, i_k\}$, let 
$$\mu_w(r) = \begin{cases}
\mathrm{min}\{m | n-k+1 \leq m \leq n, i_{n-m+1}<j_r\}\,, & \text{ if } 1 \leq r \leq n-k\,,\\
r+1\,, & \text{ if } n-k+1 \leq r \leq \lfloor n/2 \rfloor\,,
\end{cases}$$
then set
\begin{align} \label{ds}
\begin{split}
a_w(s,\chi):=& L(s+\tfrac{n+1}{2}-k, \chi)\prod_{\substack{r=1\\ i_r \geq 2r}}^{\mathrm{min}(k,\lfloor n/2 \rfloor)} L(2s+i_r-2r+1, \chi^2) \\
& \times \prod_{\substack{r=1\\ i_r \leq 2r-1}}^{\mathrm{min}(k,\lfloor n/2 \rfloor)}  L(2s-n+r+\mu_w(r)-1, \chi^2) \prod_{r=k+1}^{\lfloor n/2 \rfloor} L(2s+n+1-2r, \chi^2)\,,\\
d(s,\chi):=&L\left(s+\tfrac{n+1}{2},\chi \right)\prod_{r=1}^{\lfloor n/2 \rfloor}L(2s+n+1-2r,\chi^2)\,,\\
c_w(s,\chi):=&\frac{a_w(s,\chi)}{d(s,\chi)}\,.
\end{split}
\end{align}
We note that there is a typo in \cite{Ikeda:poles:triple}; the inequality $i_{n-m+1}<j_r$ in the definition of $\mu_w(r)$ in loc.~cit.~is reversed (see \cite[p.26]{GPSR:LNM}, which unfortunately uses different notation).
One has
\begin{align} \label{basic:cases}
\begin{split}
c_{I_{2n}}(s,\chi)&=1 \quad \textrm{ and } \quad a_{I_{2n}}(s,\chi)=d(s,\chi)\,,\\
a_{w_0}(s,\chi)&=L(s+\tfrac{1-n}{2}, \chi) \prod_{r=1}^{\lfloor n/2 \rfloor}  L(2s-n+2r, \chi^2)\,. 
\end{split}
\end{align}
  As explained below \cite[(1.2.7)]{Ikeda:poles:triple}, these $L$-factors are defined so that for non-Archimedean $F$ and unramified $\chi$ the operator $M_w$ takes the spherical vector in $I(\chi_s)$ to $c_w(s,\chi)$ times the spherical vector in $I((\chi_s)^w)$.  
Moreover, in the non-Archimedean spherical case, $d(s,\chi)$ is the smallest common denominator of the $c_w(s,\chi)$ as $w$ varies.

As explained below \cite[(1.2.7)]{Ikeda:poles:triple}, these $L$-factors are defined so that for non-Archimedean $F$ and unramified $\chi$ the operator $M_w$ takes the spherical vector in $I(\chi_s)$ to $c_w(s,\chi)$ times the spherical vector in $I((\chi_s)^w)$.  
Moreover, in the non-Archimedean spherical case, $d(s,\chi)$ is the smallest common denominator of the $c_w(s,\chi)$ as $w$ varies.

For an additive character $\psi$ we also define normalized intertwining operators
\begin{align}
M_{w_0}^*:=M_{w_0,\psi}^*:=\gamma\left(s-\tfrac{n-1}{2},\chi,\psi\right)\prod_{r=1}^{\lfloor n/2 \rfloor}\gamma(2s-n+2r,\chi^2,\psi)M_{w_0}\,.
\end{align}
Note that $\gamma(s,\chi,\psi)=\varepsilon'(s,\chi,\psi)$ in the notation of \cite{Ikeda:poles:triple}.  This notation is also used in \cite{GodementJacquetBook}.

\begin{defn} A meromorphic section $f^{(s)}$ of $I(\chi_{s})$ is a {\textbf{good section}} (of $I(\chi_{s})$) if for any $w \in \Omega_n$ the section 
$$
\frac{M_wf^{(s)}}{a_w(s,\chi)}
$$
is holomorphic.  
\end{defn}
\noindent This definition is from Ikeda \cite{Ikeda:poles:triple} (note that $a_w(s,\chi)=d(s,\chi)c_w(s,\chi)$).  We note that every holomorphic section is good \cite[Lemma 1.3]{Ikeda:poles:triple}.

In the Archimedean case we require a refinement of the notion of a good section.  To state the refinement, for real numbers $A \leq B$ and polynomials $P \in \CC[x]$ and meromorphic functions 
$f:\CC \to \CC$, let
\begin{align} \label{VAB}
V_{A,B}:=\{s \in \CC:A \leq \mathrm{Re}(s) \leq B\}\,,
\end{align}
and
\begin{align}
|f|_{A,B,P}:&=\mathrm{sup}_{s \in V_{A,B}}| P(s)f(s) |_{\mathrm{st}}\,.
\end{align}
Note that this may be $\infty$. 

Motivated by the definition of $\mathcal{L}(\tau)$ in \cite{JacquetPerfectRS}, we make the following definition:

\begin{defn}  Assume $F$ is Archimedean.
A good section $f^{(s)}$ of $I(\chi_{s})$ is an {\textbf{excellent section}} (of $I(\chi_{s})$) if for any $g \in \mathrm{Sp}_{2n}(F)$, real numbers $A<B$, and any polynomials 
$$
P_w:=P_{w,\chi} \in \CC[x]
$$ 
($w\in\{\mathrm{Id},w_0\}$) such that $P_{w}(s)a_w(s,\chi)$ has no poles for $s \in V_{A,B}$ one has
\begin{align*}
\left|M_w f^{(s)}(g)\right|_{A,B,P_w}<\infty\,.
\end{align*}
If $F$ is non-Archimedean then we say any good section is excellent.  
\end{defn}

This is a complicated definition, but it appears to be necessary. One needs control of $M_wf^{(s)}$ in vertical strips in order to define Mellin transforms of these functions. One might try to replace the $P_w$ by $P(s)/a_w(s,\chi)$ for arbitrary $P(s)$, but this turns out to be awkward because $a_w(s,\chi)$ is rapidly decreasing in vertical strips (away from its poles).  

\begin{lem} \label{lem:gamma:poly}
Let $A<B$, and for $w \in \{\mathrm{Id},w_0\}$ let $P_{w,\bar{\chi}},P_{w,\chi} \in \CC[x]$ be polynomials such that $P_{w,\bar{\chi}}(s)a_{w}(-s,\bar{\chi})$ and $P_{w,\chi}(s)a_{w}(s,\chi)$ are holomorphic and nonvanishing in $V_{A,B}$.  Then the quotients
$$
\frac{P_{\mathrm{Id},\bar{\chi}}(s)a_{\mathrm{Id}}(-s,\bar{\chi})}{P_{w_0,\chi}(s)a_{w_0}(s,\chi)}, \quad \frac{P_{w_0,\chi}(s)a_{w_0}(s,\chi)}{P_{\mathrm{Id},\bar{\chi}}(s)a_{\mathrm{Id}}(-s,\bar{\chi})},\quad \frac{P_{\mathrm{Id},\chi}(s)a_{\mathrm{Id}}(s,\chi)}{P_{w_0,\bar{\chi}}(s)a_{w_0}(-s,\bar{\chi})}, \quad \frac{P_{w_0,\bar{\chi}}(s)a_{w_0}(-s,\bar{\chi})}{P_{\mathrm{Id},\chi}(s)a_{\mathrm{Id}}(s,\chi)}
$$
are all bounded by polynomials in $s$ in $V_{A,B}$.
\end{lem}
We note that $a_{w}(s,\chi)$ is nonvanishing and has a finite number of poles in any vertical strip, so polynomials as in the lemma can always be chosen.  For the purposes of the proof we use the standard notation
\begin{align} \label{Gamma}
\Gamma_F(s):=\begin{cases} \pi^{-s/2}\Gamma\left(\tfrac{s}{2} \right) &\textrm{ if }F =\RR\\
2(2 \pi)^{-s}\Gamma(s) &\textrm{ if }F=\CC
\end{cases}\,.
\end{align}
If $F$ is real, let $\mu$ denote the sign character, and if $F$ is complex, let 
\begin{align} \label{mu:def}
\mu(z):=\frac{z}{(z\bar{z})^{1/2}}\,.
\end{align}
Here in the denominator we mean the positive square-root.
We will also use the well-known fact that every character $\chi:F^\times \to \CC^\times$ can be written uniquely as 
\begin{align} \label{explicit:chi}
\chi:=|\cdot|^{it}\mu^\alpha\,,
\end{align}
where $t \in \RR$, $\alpha \in \ZZ$ and we assume $\alpha \in \{0,1\}$ if $F$ is real.
One has
\begin{align} \label{L:id}
L(s,|\cdot|^{it}\mu^\alpha)=\Gamma_F\left(s+it+\frac{|\alpha|}{[F:\RR]}\right)
\end{align}
\cite[Appendix]{JacquetPerfectRS}.

\begin{proof}
It suffices to verify that for any $A<B$ and any polynomials $p_{\bar{\chi}},p_{\chi}$ such that $p_{\bar{\chi}}(s)L(1-s,\bar{\chi})$ and $p_{\chi}(s)L(s,\chi)$ are holomorphic and nonvanishing in $V_{A,B}$ the quotient 
\begin{align*}
\frac{p_{\bar{\chi}}(s)L(1-s,\bar{\chi})}{p_{\chi}(s)L(s,\chi)}
\end{align*}
is bounded by a polynomial for $s \in V_{A,B}$.  Write $\chi$ as in \eqref{explicit:chi}.  Then we see it suffices to show that for $s \in V_{A,B}$ with $\mathrm{Im}(s)$ large enough in a sense depending on $\chi$ that 
\begin{align} \label{quot}
\frac{\Gamma\left(\frac{1-s-it+\frac{|\alpha|}{[F:\RR]}}{2[F:\RR]^{-1}}\right)}{
\Gamma\left(\frac{s+it+\frac{|\alpha|}{[F:\RR]}}{2[F:\RR]^{-1}}\right)}
\end{align}
is bounded by a polynomial in $s$.  Recall that $\Gamma(s+1)=s\Gamma(s)$.  Thus if $F$ is complex, replacing $s$ by $s + 1$ multiplies \eqref{quot} by a rational function of $s$.  If $F$ is real, replacing $s$ by $s + 1$ has the effect of multiplying by a rational function of $s$ and replacing $\chi$ by $\chi\mu$.  In 
either case we see it suffices to assume that 
$(A,B)=(-\tfrac{1}{2},\tfrac{1}{2})$.  In this case we can apply \cite[\S III.5, Lemmas 3 and 5]{Moreno} to deduce the desired bound.
\end{proof}

\begin{lem} \label{lem:excellent} The section $f^{(s)}$ of $I(\chi_s)$ is excellent if and only if  
$M_{w_0}^*f^{(s)}$ is an excellent section of $I(\bar{\chi}_{(-s)})$.
\end{lem}

\begin{proof}
In the non-Archimedean case good sections are excellent by definition so we can apply  \cite[Lemma 1.2]{Ikeda:poles:triple}.  In the Archimedean case by the same lemma we know that $f^{(s)}$ is good if and only if $M_{w_0}^*f^{(s)}$ is good.  Let $A \leq B$.  To complete the proof in the archimedian case it suffices to check that 
\begin{align*}
\left|f^{(s)}(g)\right|_{A,B,P_{\mathrm{Id},\chi}}\quad \textrm{ and } \quad
\left| M_{w_0}f^{(s)}(g) \right|_{A,B,P_{w_0,\chi}}
\end{align*}
are finite for all $P_{w,\chi}$ such that $P_{w,\chi}a_w(s,\chi)$ is holomorphic in $V_{A,B}$ if and only if 
\begin{align*}
\left| M_{w_0}^*f^{(s)}(g)\right|_{A,B,P_{\mathrm{Id},\bar{\chi}}}\quad \textrm{ and } \quad
\left| M_{w_0}M_{w_0}^*f^{(s)}(g) \right|_{A,B,P_{w_0,\bar{\chi}}}
\end{align*}
are finite for all $P_{w,\bar{\chi}}$  such that $P_{w,\bar{\chi}}(s)a_w(-s,\bar{\chi})$ is holomorphic in $V_{A,B}$.  Here $w \in \{\mathrm{id},w_0\}$.

We note that 
\begin{align*}
& \ \gamma(s-\tfrac{n-1}{2},\chi,\psi) \prod_{r=1}^{\lfloor n/2 \rfloor}\gamma(2s-n+2r,\chi^2,\psi)\\
= & \ \frac{a_{\mathrm{Id}}(-s,\bar{\chi})}{a_{w_0}(s,\chi)}\varepsilon(s-\tfrac{n-1}{2},\chi,\psi)\prod_{r=1}^{\lfloor  n/2 \rfloor}\varepsilon(2s-n+2r,\chi^2,\psi)\,,
\end{align*}
and the $\varepsilon$ function here is a constant times $\chi_{s-\tfrac{n-2}{2}}(a) \prod_{i=1}^{\lfloor n/2 \rfloor} \chi^2_{2s-n+2r-1/2}(a)$ for some $a \in F^\times$ depending on $\psi$ (see \cite[\S 2, \S 16]{JacquetPerfectRS}).  Thus for some nonvanishing holomorphic functions $C_1,C_2$ such that $C_1(s)$ and $C_{2}(s)$ are bounded in $V_{A,B}$ one has 
\begin{align*}
P_{\mathrm{Id},\bar{\chi}}(s)M_{w_0}^*f^{(s)}&=
P_{\mathrm{Id},\bar{\chi}}(s)\gamma(s-\tfrac{n-1}{2},\chi,\psi) \prod_{r=1}^{\lfloor n/2 \rfloor}\gamma(2s-n+2r,\chi^2,\psi)M_{w_0}f^{(s)}\\
&=C_1(s)\frac{P_{\mathrm{Id},\bar{\chi}}(s)a_{\mathrm{Id}}(-s,\bar{\chi})}{P_{w_0,\chi}(s)a_{w_0}(s,\chi)}P_{w_0,\chi}(s)M_{w_0}f^{(s)}\,,
\end{align*}
and, using \cite[Lemma 1.1]{Ikeda:poles:triple}, 
\begin{align*}
P_{w_0,\bar{\chi}}(s)M_{w_0}M_{w_0}^*f^{(s)} &= \frac{\pm P_{w_0,\bar{\chi}}(s)  f^{(s)}}{\gamma(-s-\tfrac{n-1}{2},\bar{\chi},\psi) \prod_{r=1}^{\lfloor n/2 \rfloor}\gamma(-2s-n+2r,\bar{\chi}^2,\psi)}\\
 &=C_2(s)\frac{ P_{w_0,\bar{\chi}}(s)a_{w_0}(-s,\bar{\chi})}{P_{\mathrm{Id},\chi}(s)a_{\mathrm{Id}}(s,\chi)} P_{\mathrm{Id},\chi}(s)f^{(s)}\,.
 \end{align*}
The lemma now follows from Lemma \ref{lem:gamma:poly}.
\end{proof}

\section{The Schwartz space of $X$} \label{sec:Schwartz}

Let $K \leq \mathrm{Sp}_{2n}(F)$ be a maximal compact subgroup.  We now give a definition of a Schwartz space $\mathcal{S}(X(F),K) \subset C^\infty(X(F))$.  Our approach is a combination of Braverman and Kazhdan in \cite{BK:normalized} with L.~Lafforgue's approach to defining Fourier transforms using the Plancherel formula \cite{LafforgueJJM}.  The first author used a similar approach in \cite{GetzBK} to construct Schwartz spaces for Archimedean spherical functions in a different context.  

Our conventions are slightly different than those of Braverman and Kazhdan in that we work with representations induced from a single parabolic as opposed to those from two opposite parabolics.  Our reason for this is that we need the refined information obtained by Ikeda in \cite{Ikeda:poles:triple}.  Apart from this, our construction of the Fourier transform should agree with that of Braverman and Kazhdan and Braverman and Kazhdan's Schwartz space should be contained in ours, at least after normalizing by a power of $|g|$, defined as in \eqref{norm:def}.  We will not check this because it would make the current paper unnecessarily long.  One ought to be able to obtain the precise relationship using the recent preprint of Shahidi \cite{Shahidi:FT} and its appendix by Li \cite{WWLicomparison}.

For a smooth function 
$$
\Phi \in C^\infty(X(F))\,,
$$
set
\begin{align} \label{Phi:int}
\Phi_{\chi_{s}}(g)=\int_{ M^{\mathrm{ab}}(F)}\delta_P(m)^{1/2}\chi_{s}\left( \omega( m) \right)\Phi(m^{-1}g)dm\,.
\end{align}
Here we give $M^{\mathrm{ab}}(F)$ the measure induced by the isomorphism $\omega:M^{\mathrm{ab}}(F) \tilde{\to} F^\times$ and our standard choice of measure on $F^\times$ (see \S \ref{sec:notat}).
When this integral is well-defined, either because it converges absolutely or by analytic continuation in $s$ from a half plane of absolute convergence, it defines an element of the induced representation $I(\chi_s)$ of \eqref{norm:ind}.

\begin{defn}
Assume that $F$ is non-Archimedean.  The \textbf{Schwartz space} $\mathcal{S}(X(F),K)$ consists of right $K$-finite functions $\Phi \in C^\infty(X(F))$ such that for each unitary character $\chi$ and $g \in \mathrm{Sp}_{2n}(F)$ the integral \eqref{Phi:int} defining $\Phi_{\chi_{s}}(g)$ is absolutely convergent for all $s$ with $\mathrm{Re}(s)$ large enough and the map
$$
(g,s) \mapsto \Phi_{\chi_{s}}(g)
$$
is an excellent section.  
\end{defn}

Assume for the moment that $F$ is Archimedean.  
For $\Phi \in C^\infty(X(F))$ let
\begin{align} \label{D}
Df(g):=\frac{\partial}{ \partial z} \Phi(1(e^{z})g)|_{z=0}\,,
\end{align}
and when $F$ is complex,
\begin{align} \label{barD}
\bar{D}f(g):=\frac{\partial}{\partial \bar{z}} \Phi(1(e^{z})g)|_{z=0}\,.
\end{align}
Here $1(x)$ is defined as in \eqref{c:def}.

\begin{defn}
Assume that $F$ is Archimedean.  The 
 \textbf{Schwartz space} $\mathcal{S}(X(F),K)$ consists of right $K$-finite functions $\Phi \in C^\infty(X(F))$ such that for all $B \geq 0$, $B' \geq 0$ and for each unitary character $\chi$ and $g \in \mathrm{Sp}_{2n}(F)$ the integral \eqref{Phi:int} defining $(D^B \bar{D}^{B'}\Phi)_{\chi_{s}}(g)$ is absolutely convergent for all $s$ with $\mathrm{Re}(s)$ large enough and the map
$$
(g,s) \mapsto (D^B \bar{D}^{B'}\Phi)_{\chi_{s}}(g)
$$
is an excellent section.  Here, by convention, $B'=0$ if $F$ is real.
\end{defn}

In the non-Archimedean case the space $\mathcal{S}(X(F),K)$ is independent of the $\mathrm{Sp}_{2n}(F)$-conjugate $K$ of $K_0$, so in this case we are free to take $K:=K_0=\mathrm{Sp}_{2n}(\OO)$ in the proofs.

Our reason for adopting this definition of the Schwartz space is that Ikeda's work essentially tells us how the transforms $\Phi_{\chi_s}$ should behave for $\Phi$ in the Schwartz space, so we use this to reverse engineer the definition of the Schwartz space itself.  To make this precise it is useful to recall some basic facts about Mellin inversion.

Let 
\begin{align*}
I_F:&=\begin{cases}[-\tfrac{\pi}{\log q},\tfrac{\pi}{\log q}] &\textrm{ if }F \textrm{ is non-Archimedean}\\ \RR &\textrm{ if }F \textrm{ is Archimedean\,. } \end{cases}\\
c_F:&=\begin{cases}  \log q &\textrm{ if } F \textrm{ is non-Archimedean}\\
\frac{1}{2} &\textrm{ if }F=\RR\\
\tfrac{1}{2\pi} &\textrm{ if }F=\CC\,.\end{cases}
\end{align*}
Moreover, let $K_{\GG_m} \leq F^\times$ be the maximal compact subgroup.  We abuse notation and denote by $\widehat{K}_{\GG_m}$ a set of representatives for the characters of $F^\times$ modulo equivalence, 
where two characters $\eta,\eta'$ are said to be equivalent if $\eta=|\cdot|^s\eta'$ for some $s \in \CC$.  
Since by our conventions characters are unitary, we can in fact take $s \in i\RR$.
  The set of equivalence classes is in bijection with the set of characters of $K_{\GG_m}$ via restriction.  This explains the notation.

\begin{lem} \label{lem:Mellin:inv} Suppose that for all $\eta \in \widehat{K}_{\mathbb{G}_m}$ the integral defining $\Phi_{\eta_s}$ is absolutely convergent for $\mathrm{Re}(s)=\sigma$.  Suppose moreover that for $k \in K$ one has
\begin{align} \label{abs:conv}
\sum_{\eta \in \widehat{K}_{\GG_m}}  \int_{\sigma+iI_F}|\Phi_{\eta_s}(k)|_{\mathrm{st}}ds<\infty\,.
\end{align}
Then for $(m,k) \in M(F) \times K$ one has
$$
\Phi(mk)=\delta_P(m)^{1/2}\sum_{\eta \in \widehat{K}_{\GG_m}} \int_{\sigma+iI_F}\Phi_{\eta_s}(k)\eta_{s}(\omega(m)) \frac{c_F ds}{2\pi i}\,.
$$
Conversely, suppose that we are given continuous $f(\eta)^{(s)} \in I(\eta_s)$ for all $s$ with $\mathrm{Re}(s)=\sigma$ and all $\eta \in \widehat{K}_{\GG_m}$ and that 
\begin{align} \label{abs:conv2}
\sum_{\eta \in \widehat{K}_{\GG_m}}  \int_{\sigma+iI_F}|f(\eta)^{(s)}(k)|_{\mathrm{st}}ds<\infty\,.
\end{align}
Assume moreover that in the non-Archimedean case $f(\eta)^{s+\frac{2\pi i}{\log q}}=f(\eta)^s$.
Then if we define 
$$
\Phi(mk)=\delta_P(m)^{1/2}\sum_{\eta \in \widehat{K}_{\GG_m}} \int_{\sigma+iI_F}f(\eta)^{(s)}(k)\eta_{s}(\omega(m)) \frac{c_F ds}{2\pi i}\,
$$
and the integral definining $\Phi_{\eta_s}$ is absolutely convergent for all $\eta \in \widehat{K}_{\GG_m}$ and $s$ with $\mathrm{Re}(s) =\sigma$ we have
$$
\Phi_{\eta_s}=f(\eta)^{(s)}.
$$
\end{lem}

\begin{proof}
Both statements are versions of Fourier inversion (see \cite[Theorem 4.32]{FollandAHA}, for example), but one must be careful to choose the measures on $F^\times$ and its dual appropriately.  We can deduce the appropriate measures using \cite[(2.2)]{Blomer_Brumley_Ramanujan_Annals}.
\end{proof}

Recall that $w_0=w_{\{1,\dots,n\}} \in \Omega_n$ is the long Weyl element.  
\begin{thm} \label{thm:FT}  Suppose that $\Phi \in \mathcal{S}(X(F),K)$.  Then there is a unique function $\mathcal{F}(\Phi) \in \mathcal{S}(X(F),K)$ such that $\mathcal{F}(\Phi)_{\chi_{s}}=M^*_{w_0}\Phi_{\bar{\chi}_{-s}}$ for all characters $\chi$ and all $s$ with
$\mathrm{Re}(s) \geq 0$.
\end{thm}

To understand the theorem it is useful to note that 
$$
I((\chi_{s})^{w_0})=I(\bar{\chi}_{-s})\,.
$$

\begin{proof}
Using the Iwasawa decomposition write
$$
g=nmk\,,
$$
where $(n,m,k) \in N(F) \times M(F) \times K$.  We define
\begin{align} \label{FPhi}\begin{split}
\mathcal{F}(\Phi)(g):&=\sum_{\eta \in \widehat{K}_{\GG_m}}\int_{iI_F} M^*_{w_0}\Phi_{\bar{\eta}_{-s}}(g)
\frac{c_F ds}{2 \pi i}\\
&=\sum_{\eta \in \widehat{K}_{\GG_m}}\int_{iI_F} M^*_{w_0}\Phi_{\bar{\eta}_{-s}}(k)\delta_P(m)^{1/2}\eta_{s}(\omega(m))
\frac{c_F ds}{2 \pi i}\,. \end{split}
\end{align}
Note that only finitely many $\eta$ contribute a nonzero summand.  Provided that the integrals here are all absolutely convergent it is also clear that $\mathcal{F}(\Phi)$ is independent of the decomposition of $g$ into $nmk$ and it is right $K$-finite.  

By assumption $\Phi_{\chi_s}$ is an excellent section of $I(\chi_s)$ for all $\chi$, which implies that $\Phi_{\bar{\chi}_{-s}}$ is an excellent section of $I(\bar{\chi}_{-s})$ for all $\chi$ and hence $M_{w_0}^*\Phi_{\bar{\chi}_{-s}}$ is an excellent section of $I(\chi_s)$ by Lemma \ref{lem:excellent}.  

In the non-Archimedean case by definition of an excellent section we have
 $$
 \frac{M^*_{w_0}\Phi_{\bar{\chi}_{-s}}}{d(s,\chi)} \in \CC[q^{-s},q^{s}]\,.
 $$
From the description \eqref{ds} of $d(s,\chi)$ we 
see that $d(s,\chi)$ has no poles for $\mathrm{Re}(s) \geq 0$.  We deduce that each of the integrals in the definition of $\mathcal{F}(\Phi)(g)$ is absolutely convergent, so $\mathcal{F}(\Phi)$ is well-defined in this case.  We also see that \eqref{abs:conv2} holds for $f(\eta)^{(s)}=M^*_{w_0}\Phi_{\bar{\eta}_{-s}}$. 
Moreover we deduce that $\mathcal{F}(\Phi)$ is supported in 
$$
\bigcup_{c >-N} [P,P](F)c(\varpi)K_0\,,
$$
for sufficiently large $N \in \ZZ_{>0}$, and satisfies a bound of the form $|\mathcal{F}(\Phi)(mk)|_{\mathrm{st}} \ll_{\Phi} \delta_P(m)^{1/2}$.
It follows that for any $\sigma>0$ one has
\begin{align*}
&\int_{M^{\mathrm{ab}}(F)}\delta_P^{1/2}(m)|\omega(m)|^\sigma|\mathcal{F}(\Phi)(m^{-1}k)|dm\\
&\ll_{\Phi} \sum_{c>-N} \delta_P(c(\varpi))^{1/2}|\varpi^c|^{\sigma}\delta_P(c(\varpi)^{-1})^{1/2} <\infty\,.
\end{align*}
This implies that for $\sigma>0$  the integral defining $\mathcal{F}(\Phi)_{\chi_s}$ is absolutely convergent for $\mathrm{Re}(s)=\sigma$ and hence the inversion formula $\mathcal{F}(\Phi)_{\chi_s}=M^*_{w_0}\Phi_{\bar{\chi}_{-s}}(k)$ is valid.  Moreover, it is clear that $\mathcal{F}(\Phi) \in \mathcal{S}(X(F),K)$.

Now consider the Archimedean case.  As noted above  $M^*_{w_0}\Phi_{\bar{\chi}_{-s}}$ is an excellent section of $I(\chi_s)$ for all characters $\chi$. Thus for all $A < B$ and $P_{\mathrm{Id}} \in \CC[x]$ such that $P_{\mathrm{Id}}(s)d(s,\chi)$ is holomorphic in $V_{A,B}$  one has
\begin{align}
\left|M_{w_0}^* 
\Phi_{\bar{\chi}_{-s}}(g)\right|_{A,B,P_{\mathrm{Id}}} <\infty\,. \label{esti}
\end{align}
As before, from the description \eqref{ds} of $d(s,\chi)$ we 
see that $d(s,\chi)$ has no poles for $\mathrm{Re}(s) \geq 0$.  Thus we deduce that the integrals in the definition \eqref{FPhi} of $\mathcal{F}(\Phi)$ are absolutely convergent.  We also see that \eqref{abs:conv2} holds for $f(\eta)^{(s)}=M^*_{w_0}\Phi_{\bar{\eta}_{-s}}$. 

We must check that $\mathcal{F}(\Phi)$ is smooth.  In fact, since intertwining operators preserve smoothness and the Mellin transform only involves an integration over $M^{\mathrm{ab}}$ to prove smoothness of $\mathcal{F}(\Phi)$ it suffices to prove that for all $B \geq 0, B' \geq 0$ that the derivative $D^B\bar{D}^{B'}\mathcal{F}(\Phi)$ exists. Differentiating under the integral sign we see that 
\begin{align}
D^B\bar{D}^{B'}\mathcal{F}(\Phi)(g)
=\sum_{\eta \in \widehat{K}_{\GG_m}}\int_{iI_F}P_{\eta}(s) M^*_{w_0}\Phi_{\bar{\eta}_{-s}}(k)\delta_P(m)^{1/2}\eta_{s}(\omega(m))
\frac{c_F ds}{2 \pi i}
\end{align}
for some polynomials $P_{\eta}(s)$.  Since $M_{w_0}^*(\Phi)_{\bar{\eta}_{-s}}$ is an excellent section this integral converges absolutely and we deduce that $D^B\bar{D}^{B'}\mathcal{F}(\Phi)$ exists, hence $\mathcal{F}(\Phi)$ is smooth.  

The final thing that must be checked is that for all $\chi$ the integral defining $(D^B\bar{D}^{B'}\mathcal{F}(\Phi))_{\chi_s}(g)$ absolutely convergent for $\mathrm{Re}(s)$ large enough and is an excellent section.  Indeed, this implies the inversion formula $\mathcal{F}(\Phi)_{\chi_s}=M_{w_0}^*\Phi_{\bar{\chi}_{-s}}$ is valid by Lemma \ref{lem:Mellin:inv}. 

For $N \geq 0$ consider 
\begin{align*}
 &D^B\bar{D}^{B'}\mathcal{F}(\Phi)(mk)\omega\bar{\omega}(m)^{-N}\\
 &=
 \sum_{\eta \in \widehat{K}_{\GG_m}}\int_{iI_F}P_{\eta}(s) M^*_{w_0}\Phi_{\bar{\eta}_{-s}}(k)\delta_P(m)^{1/2}\eta_{s-2N[F:\RR]^{-1}}(\omega(m))
\frac{c_F ds}{2 \pi i}\\
&=\sum_{\eta \in \widehat{K}_{\GG_m}}\int_{iI_F-i2N[F:\RR]^{-1}}P_{\eta}(s+2N[F:\RR]^{-1}) M^*_{w_0}\Phi_{\bar{\eta}_{-s-2[N:\RR]^{-1}}}(k)\delta_P(m)^{1/2}\eta_{s}(\omega(m))
\frac{c_F ds}{2 \pi i}
\end{align*}
where the bar denotes complex conjugation (which is trivial if $F$ is real) and the polynomials $P_\eta$ are defined as above.  We now shift the contour to $iI_F$ to arrive at 
\begin{align*}
\sum_{\eta \in \widehat{K}_{\GG_m}}\int_{iI_F}P_{\eta}(s+2N[F:\RR]^{-1}) M^*_{w_0}\Phi_{\bar{\eta}_{-s-2[N:\RR]^{-1}}}(k)\delta_P(m)^{1/2}\eta_{s}(\omega(m))
\frac{c_F ds}{2 \pi i}
\end{align*}
This shift is permissible by the definition of an excellent section, and since $d(s,\chi)$ has no poles for $\mathrm{Re}(s) \geq 0$ we pass no poles in this process.  Now again by the definition of an excellent section we have that the above is bounded by a constant depending on $N,B,B',\Phi$ but not $m$ since $s \in iI_F$.  Thus
$$
D^B\bar{D}^{B'}\mathcal{F}(\Phi)(mk) \ll_{B,B',N,\Phi} (\omega\bar{\omega})^N(m).
$$
But then
\begin{align*}
&\int_{M^{\mathrm{ab}}(F)} \delta^{1/2}_P(m)|\omega(m)|^{\sigma}|D^BD^{B'}\mathcal{F}(\Phi)(m^{-1}k)|dm
\\&\ll \int_{M^{\mathrm{ab}}(F): |\omega(m)| \leq 1} \delta^{1/2}_P(m)|\omega(m)|^{\sigma}D^BD^{B'}\mathcal{F}(\Phi)(m^{-1}k)dm \\&+
\int_{M^{\mathrm{ab}}(F): |\omega(m)|>1} \delta^{1/2}_P(m)|\omega(m)|^{\sigma}(\omega \bar{\omega})^{-N}(m)dm
\end{align*}
which converges for $N$ and $\sigma$ large enough.  
\end{proof}

The previous theorem provides us with a Fourier transform
\begin{align} \label{FT}
\mathcal{F}:=\mathcal{F}_{\psi,K}:\mathcal{S}(X(F),K) \lto \mathcal{S}(X(F),K)\,.
\end{align}
   The Fourier transform is defined so that for every character $\chi:F^\times \to \CC^\times$ the diagram
\begin{align*}
\begin{CD}
\mathcal{S}(X(F),K) @>{\mathcal{F}}>> \mathcal{S}(X(F),K)\\
@V{\Phi \mapsto \Phi_{\chi_s}}VV @VV{\Phi \mapsto \Phi_{\bar{\chi}_{-s}}}V\\
I(\chi_s) @>{M_{w_0}^*}>> I((\chi_{s})^{w_0})
\end{CD}
\end{align*}
commutes.

For $\Phi \in \mathcal{S}(X(F),K)$ and $g \in \mathrm{Sp}_{2n}(F)$ let
$$
R(g)\Phi(x):=\Phi(xg)\,.
$$
\begin{lem}
For $\Phi \in \mathcal{S}(X(F),K)$ and $g \in \mathrm{Sp}_{2n}(F)$ the function 
$$
R(g)\Phi \in \mathcal{S}(X(F),gKg^{-1})
$$ 
and
$$
\mathcal{F}(R(g)\Phi)=R(g)\mathcal{F}(\Phi)\,, 
$$
or more precisely
$$
\mathcal{F}_{gKg^{-1}}(R(g)\Phi)=R(g)\mathcal{F}_K(\Phi)\,.
$$
\end{lem}

\begin{proof}
The map 
\begin{align*}
\bigcup_{h \in \mathrm{Sp}_{2n}(F)} \mathcal{S}(X(F),hKh^{-1}) &\lto I(\chi_s)\\
\Phi &\longmapsto \Phi_{\chi_s}
\end{align*}
is $G(F)$-equivariant for all $\chi$ and $s$, as is the intertwining map $M_{w_0}^*$.
\end{proof}

\begin{lem}
For $\Phi \in \mathcal{S}(X(F),K)$ one has $\mathcal{F}_{\bar{\psi}} \circ \mathcal{F}_{\psi}(\Phi)(g)=\Phi(g)$.
\end{lem}

\begin{proof} This follows from the identity $M_{w_0,\bar{\psi}}^* \circ M_{w_0,\psi}^*=\mathrm{Id}$ (see \cite[Lemma 1.1]{Ikeda:poles:triple}).
\end{proof}

Let $C_c^\infty(X(F),K) \subseteq C^\infty_c(X(F))$ denote the subspace of right $K$-finite functions.  Of course, in the non-Archimedean case, every element of $C_c^\infty(X(F))$ is right $K$-finite.  To construct elements in the Schwartz space the following proposition is useful:
\begin{prop} \label{prop:compact}
One has $C_c^\infty(X(F),K) \leq \mathcal{S}(X(F),K)$.
\end{prop}
\begin{proof}
Assume first that $F$ is non-Archimedean.  Then
for 
$\Phi \in C_c^\infty(X(F),K)$ and every $g \in \mathrm{Sp}_{2n}(F)$
the function 
$s \mapsto \Phi_{\chi_s}(g)$ is in 
$\CC[q^{-s},q^s]$.  Applying \cite[Lemma 1.3]{Ikeda:poles:triple} we deduce that $\Phi_{\chi_s}$ is a good section.  The proposition follows in this case.  

We defer the Archimedean case to Appendix \ref{App}.
\end{proof}

There is no circularity in deferring the Archimedean case to Appendix \ref{App} because this proposition is not used in the remainder of the paper.  Though it is not used, it will be important in applications.  Indeed, without it one does not know that the Archimedean Schwartz space is nonempty.

\section{Analytic control of the Schwartz space} \label{sec:control}

One has good analytic control of elements in the Schwartz space.  We explain this in the non-Archimedean and Archimedean settings in this section.  In applications (and even in the derivation of Theorem \ref{thm:intro}) this analytic control is vital.  To understand what is going on, it is useful to keep in mind the following toy model of the question we are answering: How does one understand a function (and, in the Archimedean case, its derivatives) given knowledge of the Mellin transform of the function?  It is well-known how to do this for functions on $\RR$, and we adapt these arguments to prove the results of this section.

In this section we will make use of the function $|\cdot|:X(F) \to \RR_{>0}$ defined in \eqref{norm:def} using the Pl\"ucker embedding and the character $\omega:P \to \GG_m$ defined as in \eqref{omega0}.

\subsection{The non-Archimedean case} \label{ssec:nonarch}

Assume for this subsection that $F$ is non-Archimedean.  
The following lemma is the analogue of \cite[Conjecture 5.6]{BK:normalized} in our setting:

\begin{lem} \label{lem:bounded}
For $\Phi \in \mathcal{S}(X(F),K_0)$ one has
\begin{align}
|\Phi(g)|_{\mathrm{st}} \ll_{\Phi} |g|^{-(n+1)/2}\,.
\end{align}
The support of $\Phi$ is contained in 
$$
\bigcup_{c>-N}[P,P](F)c(\varpi)K_0
$$
for sufficiently large $N$ (depending on $\Phi$).  

\end{lem}

\begin{proof}
Let $(m,k) \in M(F) \times K_0$.  We first show that 
$$
|\Phi(mk)|_{\mathrm{st}} \ll_{\Phi} \delta_P(m)^{1/2}\,.
$$
This and the Iwasawa decomposition
 imply the bound in the lemma because $|mk|^{-(n+1)}=\delta_P(m)$.

Since $\Phi_{\chi_s}$ is a good section for all characters $\chi$ and $g \in \mathrm{Sp}_{2n}(F)$ one has that $\frac{\Phi_{\chi_s}(g)}{d(s,\chi)}$ is a polynomial in $q^{-s}$ and $q^{s}$. 
Moreover, using \eqref{ds} we see that $d(s,\chi)$ has no poles for $\mathrm{Re}(s)>-\tfrac{1}{2}$.  It follows that \eqref{abs:conv} is valid for $\sigma=0$, and thus by Lemma \ref{lem:Mellin:inv} one has
\begin{align}
\Phi(mk)&=\sum_{\eta \in \widehat{\OO^\times}}\int_{iI_F} \Phi_{\eta_s}(k)\delta_P(m)^{1/2}\eta_{s}(\omega(m))
\frac{c_F ds }{2 \pi i}\,. \label{tag}
\end{align}  
Since $\Phi$ is left $K_0$-finite the sum on $\eta$ here
has finite support.  In addition the integral is over a compact set so we deduce that 
$|\Phi(mk)| \ll_{\Phi} \delta_P(m)^{1/2}$.

As mentioned above for fixed $g$ the function
$$ 
\frac{\Phi_{\chi_{s}}(g)}{d(s,\chi)} \in \CC[q^{s},q^{-s}]\,.
$$
 Since $d(s,\chi)^{-1}$ is a polynomial in $q^{-s}$ by \eqref{ds} we deduce that if one expands $\Phi_{\chi_{s}}(g)$ as a Laurent series in $q^{-s}$ there are only finitely many terms with negative exponent.  The second claim of the lemma thus follows from the inversion formula \eqref{tag}.
\end{proof}

Assume that $F$ is non-Archimedean.  The \textbf{basic function} on $X(F)$ is 
\begin{align} \label{basic}
b:=\sum_{a=0}^\infty 
\sum_{b_1=0}^\infty\cdots \sum_{b_{\lfloor n/2 \rfloor}=0}^\infty q^{2b_1+4b_2+\dots+2\lfloor n/2 \rfloor b_{\lfloor n/2 \rfloor} }\one_{a+2b_1+\dots+2b_{\lfloor n/2 \rfloor }}\,.
\end{align}

\begin{lem}  \label{lem:basic:comp} The function $b$ is the unique right $K_0$-invariant function on $\mathrm{Sp}_{2n}(F)$ such that 
\begin{align*}
b_{\chi_{s}}=d(s,\chi)(\one_0)_{\chi_{s}}\,,
\end{align*}
for unramified characters $\chi$ and all $s$ sufficiently large.
Here $\one_c$ is defined as in \eqref{1c}.
\end{lem}
\begin{proof}
One has
\begin{align*}
(\one_{c})_{\chi_s}(g):&=
\int_{M^{\mathrm{ab}}(F)} 
\delta_P(m)^{1/2}\chi_s(\omega(m))\one_c(m^{-1}g)dm\\
&=\int_{M^{\mathrm{ab}}(F)} \delta_P(c(\varpi)^{-1}m)^{1/2}\chi_s(\omega(c(\varpi)^{-1}m))\one_c(c(\varpi)m^{-1}g)dm\\
&=\delta_P(c(\varpi)^{-1})^{1/2}\chi_s^{-1}(\omega(c(\varpi))) \int_{M^{\mathrm{ab}}(F)} \delta_P(m)^{1/2}\chi_s(\omega(m))\one_0(m^{-1} g)dm\\
&=\delta_P(c(\varpi)^{-1})^{1/2}\chi_s(\varpi^c) \int_{M^{\mathrm{ab}}(F)} \delta_P(m)^{1/2}\chi_s(\omega(m))\one_0(m^{-1}g)dm\\
&=\delta_P(c(\varpi)^{-1})^{1/2}\chi_s(\varpi^c) (\one_0)_{\chi_s}(g)\,.
\end{align*}
Here we have used the fact that $\omega(c(x))=x^{-c}$.  
On the other hand $\delta_P(c(\varpi)^{-1})=|\varpi^{c(n+1)}|=q^{-c(n+1)}$, 
so the above is 
\begin{align*}
\chi_{s+(n+1)/2}(\varpi^c) (\one_0)_{\chi_s}(g)\,.
\end{align*}
From this computation one deduces that 
$b_{\chi_s}=d(s,\chi)(\one_0)_{\chi_s}$ as claimed.  The fact that 
$b$ is the unique function with this property follows from 
Mellin inversion (see the proof of Lemma \ref{lem:bounded}).
\end{proof}

\begin{lem} \label{lem:b:bound} The function $b$ is in $\mathcal{S}(X(F),K_0)$.  For $|\gamma|>1$ one has $b(\gamma)=0$.  If $\varepsilon>0$ and $|\gamma|\leq 1$ then for $q$ large enough in a sense depending on $\varepsilon$ one has $
|b(\gamma)| \leq |\gamma|^{-(n+1)/2-\varepsilon}$.  
\end{lem}

\begin{proof}
To prove that $b \in \mathcal{S}(X(F),K)$ we must show that $b_{\chi_s}$ is a good section of $I(\chi_s)$ for all characters $\chi:F^\times \to \CC^\times$.  This is stated right above \cite[Lemma 1.2]{Ikeda:poles:triple}.

For the last assertion we note that by Mellin inversion
\begin{align*}
b(mk)&=\int_{iI_F+\varepsilon} d(s,\chi)(\one_0)_{1_s}(k)\delta_P(m)^{1/2}\eta_{s}(\omega(m))
\frac{c_F ds }{2 \pi i}\\
&=\int_{iI_F+\varepsilon} d(s,\chi)\delta_P(m)^{1/2}\eta_{s}(\omega(m))
\frac{c_F ds }{2 \pi i}\,.
\end{align*}
This in turn is bounded in absolute value by $|d(\varepsilon,1)|_{\mathrm{st}}\delta_P(m)^{1/2}|\omega(m)|^{\varepsilon}$.  The function $d(s,1)$ is a local factor of a product of global $L$-functions that converge absolutely at $s=\varepsilon$, so for $q$ sufficiently large $|d(\varepsilon,1)|_{\mathrm{st}} \leq 1$.
\end{proof}

\begin{lem} \label{lem:basic:fixed}  Assume that $\psi$ is unramified.  Then one has $\mathcal{F}(b)=b$.
\end{lem} 

\begin{proof}
Essentially by definition (see below \cite[(1.2.7)]{Ikeda:poles:triple}) one has
\begin{align*}
M_{w_0}b_{\chi_{s}}&=c_{w_0}(s,\chi)d(s,\chi)(\one_{0})_{\bar{\chi}_{-s}}\\
&=a_{w_0}(s,\chi)(\one_{0})_{\bar{\chi}_{-s}}\,.
\end{align*}
We note in particular that $b_{\chi_s}$ vanishes unless $\chi$ is unramified.  Now
\begin{align*}
\gamma\left(s-\tfrac{n-1}{2},\chi,\psi\right)\prod_{r=1}^{\lfloor n/2 \rfloor}\gamma(2s-n+2r,\chi^2,\psi)
&=\frac{L\left(\tfrac{n+1}{2}-s,\bar{\chi}\right)}{L\left(s-\tfrac{n-1}{2},\chi\right)}\prod_{r=1}^{\lfloor n/2 \rfloor}\frac{L(1+n-2r-2s,\overline{\chi}^2)}{L(2s-n+2r,\chi^2)}\\
&=\frac{d(-s,\bar{\chi})}{a_{w_0}(s,\chi)}\,,
\end{align*}
so 
$$
M_{w_0}^*b_{\chi_{s}}=d(-s,\bar{\chi})(\one_{0})_{\bar{\chi}_{-s}}\,.
$$
The lemma follows.
\end{proof}

\subsection{The Archimedean case}
\label{ssec:arch}

In this subsection we assume $F$ is Archimedean.
The usual Schwartz space $\mathcal{S}(\RR)$ of $\RR$ enjoys the following properties:
\begin{enumerate}
\item It is closed under multiplication by polynomials.
\item Its elements are bounded.
\end{enumerate}
One often says loosely that functions in $\mathcal{S}(\RR)$ are rapidly decreasing at infinity, which follows upon combining 
these two assertions.  We prove the analogues of (1-2) for functions in $\mathcal{S}(X(F),K)$ in this section.  

Using the Iwasawa decomposition $G(F)=P(F)K$ define 
\begin{align} \label{omega}
\begin{split}
\omega \bar{\omega}:G(F) \lto \RR_{>0}\\
pk \longmapsto \omega\bar{\omega}(p)\,.
\end{split}
\end{align}
Here the bar denotes complex conjugation, which we take to be trivial if $F$ is real (so in this case $\omega \bar{\omega}=\omega^2$).  It is easy to see that this is well-defined, which is the reason we chose to work with $\omega \bar{\omega}$ instead of $\omega$.  

The following lemma is a weak analogue of property (1):
\begin{lem} \label{lem:char} For all $\alpha \in \ZZ_{ \geq 0}$ if $f \in \mathcal{S}(X(F),K)$ then
$$
\omega^{-\alpha} \bar{\omega}^{-\alpha} f \in \mathcal{S}(X(F),K)\,.
$$
\end{lem}

Before proving this lemma it is useful to first prove a result on Archimedean $L$-factors.  
Recall the character $\mu$ from \eqref{mu:def}.

\begin{lem} \label{lem:shift}
Let $\chi:F^\times \to \CC^\times$ be a character. The quotients
$$
\frac{L(s+[F:\RR]^{-1},\chi\mu)}{L(s,\chi)} \quad{ and }\quad
 \frac{L(s+[F:\RR]^{-1},\chi\bar{\mu})}{L(s,\chi)}
$$
are polynomials in $s$ of degree $\leq 1$ that are bounded by a polynomial of degree $1$ in $s$.
\end{lem}

Here if $F$ is real $\bar{\mu}$ is just $\mu$, so in this case the second identity is redundant. 
\begin{proof}
Write $\chi$ as in \eqref{explicit:chi}.  Using \eqref{L:id} we see that if $F$ is real
\begin{align*}
\frac{L(s+1,|\cdot|^{it}\mu^\alpha\mu)}{L(s,|\cdot|^{it}\mu^\alpha)}=\begin{cases}  \frac{s+it}{2\pi}&\textrm{ if }\alpha=0\\
1 & \textrm{ if }\alpha=1\end{cases}\,.
\end{align*}
Similarly, if $F$ is complex, 
\begin{align*}
\frac{L(s+\tfrac{1}{2},|\cdot|^{it}\mu^\alpha \mu)}{L(s,|\cdot|^{it}\mu^\alpha)}&=\begin{cases} \frac{s+it +|\alpha|/2}{2 \pi} &\textrm{ if }\alpha \geq 0\\
1 & \textrm{ if }\alpha<0\end{cases}\,,\\
\frac{L(s+\tfrac{1}{2},|\cdot|^{it}\mu^\alpha\bar{\mu})}{L(s,|\cdot|^{it}\mu^\alpha)}&=\begin{cases} \frac{s+it+|\alpha|/2}{2 \pi} &\textrm{ if }\alpha \leq 0\\
1 & \textrm{ if }\alpha>0\end{cases}\,.
\end{align*}
The lemma follows.
\end{proof}

\begin{proof}[Proof of Lemma \ref{lem:char}]
By induction it suffices to treat the case where $\alpha=1$.
For $m \in M(F)$, $k \in K$ one has
\begin{align*}
(\omega^{-1}\bar{\omega}^{-1} \Phi)_{\chi_s}(k)&=\int_{M^{\mathrm{ab}}(F)} \delta_P(m)^{1/2} \chi_{s}(\omega(m))\omega(m)\bar{\omega}(m)\Phi(m^{-1}k)dm\\
&=\Phi_{\chi_{s+2[F:\RR]^{-1}}}(k)\,.
\end{align*}
By definition, excellent sections are good, so, for all $w \in \Omega_n$ one has
$$
\frac{M_{w}\Phi_{\chi_{s+2[F:\RR]^{-1}}}(k)}{a_{w}(s+2[F:\RR]^{-1},\chi)} \in E\,.
$$

By Lemma \ref{lem:shift} and the definition of $a_{w}(s,\chi)$
$$
\frac{a_{w}(s+2[F:\RR]^{-1},\chi)}{a_{w}(s,\chi)} 
$$
is a polynomial in $s$.  It follows that 
$$
\frac{M_{w}(\omega^{-1}\bar{\omega}^{-1} \Phi)_{\chi_s}}{a_{w}(s,\chi)} \in E\,,
$$ 
for all $w \in \Omega_n$.  Hence $(\omega^{-1}\bar{\omega}^{-1} \Phi)_{\chi_s}$ is a good section for all characters $\chi:F^\times \to \CC^\times$ and $s \in \CC$.  One sees similarly that it is moreover an excellent section.
\end{proof}

The following lemma is the analogue of property (2).  

\begin{lem} \label{lem:int:by:parts} For any $\Phi \in \mathcal{S}(X(F),K)$ and $N \in \ZZ_{\geq 0}$
$$
|\Phi(g)|_{\mathrm{st}} \ll_{\Phi,N} |g|^{-(n+1)/2-N}\,.
$$ 
\end{lem}

\begin{proof}  By the Iwasawa decomposition it
 suffices to verify that 
 \begin{align}
 |\Phi(mk)|_{\mathrm{st}} \ll_{\Phi,N} |m|^{-(n+1)/2-N}\,.
 \end{align}

By Mellin inversion (see Lemma \ref{lem:Mellin:inv}) we have
$$
\Phi(mk)=\delta_P(m)^{1/2}\sum_{\eta \in \widehat{K}_{\GG_m}}\int_{\sigma+iI_F}\Phi_{\eta_{s}}(k)\eta_{s}(\omega(m))\frac{c_F ds}{2 \pi i}\,,
$$
for $\sigma$ large enough.
Now $d(s,\chi)$ is holomorphic and bounded on the line $\mathrm{Re}(s)=0$.  Therefore using the definition of excellent sections we see that we can take $\sigma=0$, and moreover that for each $\eta$ the integral over $iI_F$ in this expression is absolutely convergent.  Since the sum on $\eta$ has finite support we deduce that 
$$
|\Phi(mk)|_{\mathrm{st}} \ll_{\Phi} \delta_P(m)^{1/2}=|m|^{-(n+1)/2}\,.
$$
On the other hand by Lemma \ref{lem:char} one has
$$
|(\omega(m)\bar{\omega}(m))^{-N}\Phi(mk)|_{\mathrm{st}} \ll_{\Phi,N} \delta_P(m)^{1/2}\,,
$$
for all integers $N$, and $|m|=|\omega(m)^{-1}|$.
\end{proof}

We end this section by computing the effect of the differential operators \eqref{D} and \eqref{barD} on the Mellin transforms of a function in $\mathcal{S}(X(F),K)$.  We start with the following version of integration by parts:

\begin{lem} \label{lem:ibp}  Let $f_1,f_2:F^\times \to \CC$ be smooth functions such that $f_1(x)f_2(x) \to 0$ as $x \to \infty$. 

When $F$ is real, assume that $f_1(x)f_2(x) \to 0$ as $x \to 0$ and when $F$ is complex, assume that $f_1(x)f_2(x)\bar{x}^{-1}$ extends to a smooth function on $F$.  If
$$
 \frac{\partial}{\partial z}f_1(e^zx)|_{z=0}f_2(x), \quad f_1(x)\frac{\partial}{\partial z}f_2(e^z x)|_{z=0} \in L^1(F^\times,dx^\times)\,.
$$  
Then
$$
\int_{F^\times}\frac{\partial}{ \partial z}f_1(e^zx)|_{z=0}f_2(x)dx^\times=-\int_{F^\times}f_1(x)\frac{\partial}{\partial z}f_2(e^z x)|_{z=0} dx^\times\,.
$$

Assume now that $F$ is complex and $f_1(x)f_2(x)x^{-1}$ extends to a smooth function on $F$.  If
$$
\frac{\partial}{\partial \bar{z}}f_1(e^{z}x)|_{z=0}f_2(x),\quad f_1(x)\frac{\partial}{\partial \bar{z}}f_2(e^{z} x)|_{z=0} \in L^1(F^\times,dx^\times)
$$ 
then
\begin{align*}
\int_{F^\times}\frac{\partial }{\partial \bar{z}}f_1(e^{z}x)|_{z=0}f_2(x)dx^\times=-\int_{F^\times}f_1(x)\frac{\partial}{\partial \bar{z}}f_2(e^{z} x)|_{z=0} dx^\times\,.
\end{align*}
\end{lem}
For the reader's convenience we include the (easy) proof.
\begin{proof}
By the product rule and its analogue for Wirtinger derivatives in the complex case one has
\begin{align*}
x\frac{\partial}{\partial x}(f_1f_2)(x)=\frac{\partial}{\partial z}(f_1f_2)(e^{z}x)|_{z=0}=\frac{\partial}{\partial z}f_1(e^zx)|_{z=0}f_2(x)+f_1(x)\frac{\partial}{\partial z}f_2(e^zx)|_{z=0}\\
\bar{x}\frac{\partial}{\partial \bar{x}}(f_1f_2)(x)=\frac{\partial}{\partial \bar{z}}(f_1f_2)(e^{z}x)|_{z=0}=\frac{\partial}{\partial \bar{z}}f_1(e^zx)|_{z=0}f_2(x)+f_1(x)\frac{\partial}{\partial \bar{z}}f_2(e^zx)|_{z=0}\,.
\end{align*}
Since we have assumed that each summand on the right hand side of these two equalities is in $L^1(F^\times,dx^\times)$ the left hand side is as well.  It suffices to verify that 
\begin{align*}
\int_{F^\times}z\frac{\partial}{\partial z}(f_1f_2)(z) dz^\times \textrm{ and }\int_{F^\times} \bar{ z} \frac{\partial}{\partial \bar{z}}(f_1f_2)(z) dz^\times
\end{align*}
are zero.  When $F$ is real the integral on the left is $\zeta(1)$ times
\begin{align*}
-(f_1(0)f_2(0)-f_1(-\infty)f_2(-\infty))+(f_1(\infty)f_2(\infty)-f_1(0)f_2(0))\,,
\end{align*}
and every term here is zero by assumption.  This completes the proof in this case.  

If $F$ is complex then by Green's theorem we have 
\begin{align*}
\int_{F^\times}z\frac{\partial}{\partial z}(f_1f_2)(z) dz^\times &=\zeta(1)i\lim_{r \to \infty} \oint_{C_r} f_1(z)f_2(z)\frac{dx-idy}{\bar{z}}=0\\
\int_{F^\times}\bar{z}\frac{\partial}{\partial \bar{z}}(f_1f_2)(z) dz^\times& =-\zeta(1)i\lim_{r \to \infty} \oint_{C_r} f_1(z)f_2(z)\frac{dx+idy}{z}=0\,,
\end{align*}
where $C_r$ is the circle of radius $r$ centered at $0$ and the line integral is taken in a  counterclockwise direction.  Here we have used our conventions on Haar measures given in \S \ref{sec:notat}.
\end{proof}

\begin{lem} \label{lem:DO} 
Write
$\chi=|\cdot|^{it}\mu$ as in \eqref{explicit:chi}.  Then
\begin{align*}
(D\Phi)_{\chi_s}=\begin{cases}
(it+s+\tfrac{n+1}{2})\Phi_{\chi_s} &\textrm{ if }F \textrm{ is real}\,,\\
(\alpha/2+it+s+\tfrac{n+1}{2})\Phi_{\chi_s} &\textrm{ if }F \textrm{ is complex\,.} \end{cases}
\end{align*}
If $F$ is complex then
\begin{align*}
(\bar{D}\Phi)_{\chi_s}=
(-\alpha/2+it+s+\tfrac{n+1}{2})\Phi_{\chi_s}\,.
\end{align*}

\end{lem}

\begin{proof}
One has 
\begin{align*}
(D\Phi)_{\chi_s}(g)=\int_{M^{\mathrm{ab}}(F)}\delta_P(m)^{1/2} \chi_s(\omega(m))\frac{\partial }{\partial z}\Phi(1(e^z)m^{-1}g)|_{z=0}dm\,,
\end{align*}
for $\mathrm{Re}(s)$ sufficiently large.
Applying Lemma \ref{lem:ibp} we see that this is equal to 
\begin{align} \label{int:after:ibp}
-\int_{M^{\mathrm{ab}}(F)}\frac{\partial}{\partial z}\left(\delta_P(1(e^{z}))^{-1/2}(\bar{\chi})_{-s}(e^{z})\right)|_{z=0}\delta_P(m)^{1/2} \chi_s(\omega(m))\Phi(m^{-1}g)dm\,,
\end{align}
for $s$ in the same range.  The lemma follows upon computing the derivative.  The proof for $D$ replaced by $\bar{D}$ is similar.
\end{proof}

We will not need it until the proof of Theorem \ref{thm:Eis:bound} below, but for the reader's convenience we recall the definition of the analytic conductor of $\chi_s$ for characters $\chi:F^\times \to \CC$.  If $\chi$ is as in \eqref{explicit:chi} then 
\begin{align} \label{AC:def}
C(\chi_s):=1+\left|s+it+\frac{|\alpha|}{[F:\RR]}\right|_{\mathrm{st}}
\end{align}
(a convenient reference is \cite[\S 1]{BrumleyNarrow}).

\section{The global summation formula}
\label{sec:global:sum}
In this section $F$ is a number field with ring of integers $\OO$.  Let $K \leq \mathrm{Sp}_{2n}(\A_F)$ be a maximal compact subgroup such that $K^\infty$ is $\mathrm{Sp}_{2n}(\A_F^\infty)$-conjugate to $\mathrm{Sp}_{2n}(\widehat{\OO})$. We let
$$
\mathcal{S}(X(\A_F),K):=\prod_v'\mathcal{S}(X(F_v),K_v)
$$
be the restricted direct product with respect to the basic functions $b_v$ of \S \ref{ssec:nonarch}.  
We define an adelic Fourier transform
\begin{align}
\mathcal{F}:=\mathcal{F}_{\psi,K}:\mathcal{S}(X(\A_F),K) \lto \mathcal{S}(X(\A_F),K)
\end{align}
by taking the tensor product of the local Fourier transforms.  This is well-defined because the Fourier transform takes the basic function to the basic function at almost every place by Lemma \ref{lem:basic:fixed}.

Let $A_{\GG_m} \leq F_\infty^\times$ be the diagonal copy of $\RR_{>0}$, let
$$
[\GG_m]:=A_{\GG_m} F^\times \backslash \A_F^\times\,,
$$
and let $
\widehat{[\GG_m]}
$
be the group of characters of $[\GG_m]$.  For $\chi \in \widehat{[\GG_m]}$ and $\Phi \in \mathcal{S}(X(\A_F),K)$ let
$$
\Phi_{\chi_s}(g):=\int_{ M^{\mathrm{ab}}(\A_F)}\delta_P(m)^{1/2}\chi_{s}\left( \omega( m) \right)\Phi(m^{-1}g)dm
$$
(this is the adelic analogue of \eqref{Phi:int}).

One then obtains an Eisenstein series
\begin{align}
E(g,\Phi_{\chi_{s}}):=\sum_{\gamma \in P(F) \backslash \mathrm{Sp}_{2n}(F)}\Phi_{\chi_{s}}(\gamma g)\,.
\end{align}
Though the meromorphic continuation and functional equation of this Eisenstein series are due to Langlands, and perhaps even Siegel in the special case at hand, more precise information was obtained by Ikeda (see \cite[Proposition 1.6]{Ikeda:poles:triple}):

\begin{thm}[Ikeda] \label{thm:Ikeda} The Eisenstein series $E(g,\Phi_{\chi_{s}})$ is absolutely convergent for $\mathrm{Re}(s)$ sufficiently large.  It admits a meromorphic continuation to the plane, holomorphic except for simple poles.  The poles can only occur if $\chi^2=1$.  If $\chi=1$, then the poles can only occur at 
$$
s \in \left\{ \pm (\tfrac{n+1}{2} -m): m \in \mathbb{Z}, 0 \leq m < \tfrac{n+1}{2} \right\}\,,
$$
 and if $\chi^2=1$ but $\chi \neq 1$ then the poles can only occur at $$
 s \in \left\{ \pm (\tfrac{n-1}{2} -m): m \in \mathbb{Z}, 0 \leq m < \tfrac{n-1}{2}\right\}\,.
 $$
The Eisenstein series satisfies the functional equation
\begin{align*}
E(g,\Phi_{\chi_{s}})=E(g,M_{w_0}^*(\Phi_{\chi_{s}}))\,.
\end{align*} \qed
\end{thm}

For a different family of sections this result was also obtained by Kudla and Rallis \cite{Kudla:Rallis:Festschrift,Kudla:Rallis:firstterm}.
There is one point that must be explained in deducing Theorem \ref{thm:Ikeda} from \cite[Proposition 1.6]{Ikeda:poles:triple}.  In loc.~cit.~$E(g,M_{w_0}^*(\Phi_{\chi_s}))$ 
is replaced by $E(g,M_{w_0}(\Phi_{\chi_s}))$.  However, one has the following lemma:
\begin{lem}
One has
$M_{w_0}(\Phi_{\chi_s})=M_{w_0}^*(\Phi_{\chi_s}).$
\end{lem}

\begin{proof}
We follow the proof of \cite[Lemma 1.4]{Ikeda:poles:triple}.  Let $S$ be a set of places of $F$ including the infinite places such that $\Phi_v=b_v$ for 
$v \not \in S$ and such that $\psi$ (our fixed additive character) is unramified outside of $S$ and 
$F_v$ is absolutely unramified for $v \not \in S$.
One has
\begin{align*}
M_{w_0}(\Phi_{\chi_s})=\left(\prod_{v \not \in S} d(s,\chi_v)c_{w_0}(s,\chi_v)\one_{[P,P]K_0,\bar{\chi}_{-s}} \right)\times \prod_{v \in S}M_{w_0}(\Phi_{v,\chi_s})
\end{align*}
(see below \cite[(1.2.7)]{Ikeda:poles:triple}).  This in turn is equal to 
\begin{align} \label{before:func:eq}
\left(\prod_{v \not \in S} a_{w_0}(s,\chi_v)\one_{[P,P]K_0,\bar{\chi}_{-s}} \right)\times \prod_{v \in S}M_{w_0}(\Phi_{v,\chi_s})\,.
\end{align}
By the functional equation of Hecke $L$-functions we have
\begin{align*}
\prod_{v \not \in S} a_{w_0}(s,\chi_v)
=&\ \left(L^S(\tfrac{n+1}{2}-s,\bar{\chi})\prod_{r=1}^{\lfloor n/2 \rfloor}L^S(1+n-2r-2s,\bar{\chi}^{2})\right)\\
&\ \times \left(\gamma_S(s-\tfrac{n-1}{2},\chi,\psi)\prod_{r=1}^{\lfloor n/2 \rfloor} \gamma_S(2s-n-2r, \chi^2,\psi)\right)\\
=&\ d^S(-s,\bar{\chi})\left(\gamma_S(s-\tfrac{n-1}{2},\chi,\psi)\prod_{r=1}^{\lfloor n/2 \rfloor} \gamma_S(2s-n-2r, \chi^2,\psi)\right)\,.
\end{align*}
So \eqref{before:func:eq} becomes
\begin{align*}
\left(\prod_{v \not \in S} d(-s,\bar{\chi}_v)\one_{[P,P]K_0,\bar{\chi}_{-s}} \right)\times \prod_{v \in S}M_{w_0}^*(\Phi_{v,\chi_s})\\
=\prod_{v}M_{w_0}^*(\Phi_{v,\chi_s})
\end{align*}
(see Lemma \ref{lem:basic:fixed}).
\end{proof}

\begin{thm} \label{thm:Eis:bound}
Let $P \in \CC[x]$ be any polynomial that vanishes at every pole of $E(g,\Phi_{\chi_s})$.  Then for $A \leq B$, $A \leq \mathrm{Re}(s) \leq B$, and any $N \in \ZZ_{\geq 0}$ one has an estimate
$$
|E(g,\Phi_{\chi_s})|_{A,B,P} \ll_{A,B} C(\chi_s)^{-N}\,.
$$
\end{thm}

\noindent Here in the theorem $C(\chi_s)$ is the analytic conductor of $\chi_s$. It is 
$$
|\OO/\mathfrak{f}_\chi|\prod_{v|\infty} C((\chi_{v})_s)\,,
$$
where $C((\chi_v)_{s})$ is defined as in \eqref{AC:def} for infinite places $v$ and $\mathfrak{f}_\chi \subset \OO$ is the usual conductor of $\chi$.
In fact, since $\Phi$ is fixed on the right by a compact open subgroup of $G(F)$ the function $\Phi_{\chi_s}$ vanishes identically if $\mathfrak{f}_{\chi}$ is sufficiently small.  Therefore the finite part $|\OO/\mathfrak{f}_\chi|$ of the conductor can be ignored in the proof below.

 It is convenient to first prove the following three lemmas:
\begin{lem} \label{lem:sum:bound} For $\Phi \in \mathcal{S}(X(\A_F),K)$ and $m \in M(\A_F)$ the sum
\begin{align} \label{Phi:bound0}
\sum_{\gamma \in X(F)}\Phi(m\gamma)
\end{align}
converges absolutely and uniformly on compact subsets of $M(\A_F)$.
\end{lem}

\begin{proof}
Assume without loss of generality that $K^\infty=\mathrm{Sp}_{2n}(\widehat{\OO})$.
Recall that we have defined $|g|_v=|\mathrm{Pl}(g)|_v$ for $g \in \mathrm{Sp}_{2n}(F_v)$, and 
this function is invariant under left multiplication by $N(F_v)$ and right 
multiplication by $K_v$ (see \eqref{norm:def}).
We set
$$
|g|:=\prod_{v}|g|_v\,.
$$
Let $\Omega \subset M(\A_F)$ be a compact set and let $m \in \Omega$.
By Lemmas \ref{lem:bounded}, \ref{lem:b:bound}, and \ref{lem:int:by:parts},
for any $A \in \ZZ_{\geq 0}$, $\varepsilon>0$ and $\gamma \in X(F)$ we have
\begin{align} \label{Phi:bound}
|\Phi(m\gamma)| &\ll_{\Phi,\Omega,A,\varepsilon} \prod_{v}\max(|\gamma|_v,1)^{-A}|\gamma|^{-(n+1)/2-\varepsilon}_v= \prod_{v}\max(|\gamma|_v,1)^{-A}\,. 
\end{align}
Using \eqref{Phi:bound}, Lemmas \ref{lem:bounded} and \ref{lem:b:bound}, and the Pl\"ucker embedding of \S \ref{ssec:Plucker} we deduce that \eqref{Phi:bound0} is bounded by 
$$
\sum_{\xi \in \mathfrak{N}^{-1}\wedge^n(\OO^{2n})} \prod_{v}\max(|\xi|_v,1)^{-A}\,.
$$
Here $\mathfrak{N} \in F^{\times}$ and $|\cdot|_v$ is the norm on $\wedge^nF_v^{2n}$ used to define $|\cdot|_v$ on $X(F_v)$.  It is easy to see that this sum is bounded for $A$ sufficiently large.
\end{proof}

\begin{lem} \label{lem:Eis:bound0} There is a constant $\beta_{F,n}$ depending only on $F$ and $n$ such that the sum defining $E(g,\Phi_{\chi_s})$ converges absolutely for all $\chi$ and $s$ with $\mathrm{Re}(s) >\beta_{F,n}$. For $\mathrm{Re}(s)=A>\beta_{F,n}$ one has
\begin{align}
|E(g,\Phi_{\chi_s})|_{\mathrm{st}} \ll_{\Phi,A} 1\,.
\end{align}
\end{lem}

\begin{proof}
Replacing $\Phi$ by $R(g)\Phi$ and $K$ by $gKg^{-1}$ we see that it suffices to prove the lemma in the case $g=I_{2n}$. 
One has
\begin{align*}
&|E(I_{2n},\Phi_{\chi_s})|_{\mathrm{st}} \\&\leq \sum_{\gamma \in P(F) \backslash \mathrm{Sp}_{2n}(F)} |\Phi_{\chi_s}(\gamma)|_{\mathrm{st}}\\
&\leq \sum_{\gamma \in P(F) \backslash \mathrm{Sp}_{2n}(F)}\int_{M^{\mathrm{ab}}(\A_F^\times)} \delta_P^{1/2}(m)|\omega(m)|^A|\Phi(m^{-1}\gamma)|dm \\
&=\int_{F^\times \backslash \A_F^\times}^{+}\sum_{\gamma \in X(F)} |x|^{-\tfrac{n+1}{2}-A}|\Phi(1(x)^{-1}\gamma )| dx^\times+\int_{F^\times \backslash \A_F^\times}^- \sum_{\gamma \in X(F)} |x|^{-\tfrac{n+1}{2}-A}|\Phi(1(x)^{-1}\gamma )| dx^\times \,,
\end{align*}
here $\int^+$ denotes the integral over $x$ with $|x|>1$ and $\int^-$ denotes the integral over $x$ with $|x|\leq 1$.  For $B_\pm \in \ZZ_{\geq 0}$ this is bounded by a constant depending on $\Phi, B_\pm$ times the sum of the two integrals
\begin{align*}
& \int^{\pm}_{F^\times \backslash \A_F^\times} \sum_{\gamma \in X(F)} \one_{\mathfrak{N}^{-1}\wedge^n\widehat{\OO}^{2n}}(\mathrm{Pl}(1(x)^{-1}\gamma))|x|^{-\tfrac{n+1}{2}-A}\prod_{v}\max(|1(x)^{-1}\gamma|_v,1)^{-B_\pm}|1(x)^{-1}\gamma|^{-(n+1)/2-\varepsilon}_v dx^\times
\end{align*}
for some $\mathfrak{N} \in \OO \cap F^\times$ by
Lemmas \ref{lem:bounded}, \ref{lem:b:bound}, and  \ref{lem:int:by:parts}.

One checks using the definition of $\mathrm{Pl}$ that that $|1(x)^{-1}\gamma|_v=|x|^{-1}|\gamma|_v$.  In particular each summand is invariant as a function of $x$ under multiplication by $\widehat{\OO}^{\times}$.  Let $(\A_F^\times)^1:=\mathrm{ker} |\cdot|$.  Choosing a compact measurable fundamental domain for 
$F^\times \backslash (\A_F^\times)^1/\widehat{\OO}^\times$ and integrating over it we see that the above is
\begin{align*}
&\int_{F^\times \backslash \A_F^\times}^{\pm} \sum_{\gamma \in X(F)} \one_{\mathfrak{N}^{-1}\wedge^n \widehat{\OO}^{2n}}(\mathrm{Pl}(1(x)^{-1}\gamma))|x|^{-A+\varepsilon}\prod_{v}\max(|1(x)^{-1}\gamma|_v,1)^{-B_\pm}dx^\times \\
&\ll \int_{A_{\GG_m}}^{\pm} \sum_{\gamma \in X(F)} \one_{\mathfrak{N}^{-1}\wedge^n \widehat{\OO}^{2n}}(\mathrm{Pl}(\gamma))|x|^{-A+\varepsilon}\prod_{v}\max(|1(x)^{-1}\gamma|_v,1)^{-B_\pm}dx^\times\,.
\end{align*}
  We now employ the Pl\"ucker embedding to see that after this is bounded by
\begin{align} \label{stuff:to:bound}
\int_{A_{\GG_m}}^{\pm} \sum_{\delta \in \mathfrak{N}'^{-1}\wedge^{n}\OO^{2n} - \{0\} }|x|^{-A+\varepsilon}\prod_{v}\max(|x^{-1}\delta|_v,1)^{-B_\pm}dx^\times.
\end{align}
for some $\mathfrak{N}' \in F^\times$. 
For $y\in \wedge^nF_\infty^{2n}$ let
$$
||y||_\infty:=\max_{v|\infty}|y|_v\,.
$$
Then if $||y||_\infty \geq 1$ one has $\prod_{v}\max(|y|_v,1) \geq ||y||_\infty$ whereas if $||y||_\infty <1$ one has $ \prod_{v}\max(|y|_v,1)=1$.  Motivated by this we divide \eqref{stuff:to:bound} into two terms, namely the contribution of  $||x^{-1}\delta||_\infty \leq  1$ 
and the contribution of $||x^{-1}\delta||_\infty > 1$.
Thus the sum of the $\pm$ contributions of \eqref{stuff:to:bound} is bounded by the sum of the following three terms:
\begin{align} \label{int:to:bound231}
&\int_{0}^\infty \sum_{\substack{\delta \in \mathfrak{N}'^{-1}\wedge^{n}\OO^{2n} - \{0\}\\ ||\delta ||_\infty\leq
x^{[F:\QQ]^{-1}}} }x^{-A+\varepsilon}dx^\times\,,\\ \label{int:to:bound232:-}
&\int_{0}^1 \sum_{\substack{\delta \in \mathfrak{N}'^{-1}\wedge^{n}\OO^{2n} - \{0\}\\ ||\delta||_\infty >x^{[F:\QQ]^{-1}} } }x^{-A+\varepsilon+B_-[F:\QQ]^{-1}}||\delta||_\infty^{-B_-}dx^\times\,,\\
&\int_{1}^\infty \sum_{\substack{\delta \in \mathfrak{N}'^{-1}\wedge^{n}\OO^{2n} - \{0\}\\ ||\delta||_\infty >x^{[F:\QQ]^{-1}} } }x^{-A+\varepsilon+B_+[F:\QQ]^{-1}}||\delta||_\infty^{-B_+}dx^\times\,. \label{int:to:bound232:+}
\end{align}
So it suffices to prove that for $A$ sufficiently large we can choose $B_\pm$ so that these three terms are finite.

Now there is a constant $c>0$ so that for $\delta \in \mathfrak{N}^{-1}\wedge^{n}\OO^{2n}$ one has $||\delta||_{\infty} \leq c$ if and only if $\delta=0$.  Thus the integral in \eqref{int:to:bound231} has support in the range $x>c$ for this $c$.   Thus \eqref{int:to:bound231} is equal to 
\begin{align*}
\int_{c}^\infty \sum_{\substack{\delta \in \mathfrak{N}'^{-1}\wedge^{n}\OO^{2n} - \{0\}\\ ||\delta ||_\infty\leq
x^{[F:\QQ]^{-1}}} }x^{-A+\varepsilon}dx^\times\,,
\end{align*}
for sufficiently small $c>0$.  Moreover there is an $A'>0$ such that this is bounded by 
\begin{align*}
\int_{c}^\infty x^{-A+\varepsilon+A'}dx^\times\,,
\end{align*}
which is convergent for $A>\varepsilon+A'$.  Thus if $A$ is sufficiently large \eqref{int:to:bound231} is finite.  

As for the latter two terms, start by assuming $B_\pm $ is large enough that 
\begin{align*}
\sum_{\delta \in \mathfrak{N}'^{-1} \wedge^n\OO^{2n}-0}||\delta||_\infty^{-B_{\pm}}
\end{align*}
converges.  Under this assumption if $A>\varepsilon+B_+[F:\QQ]^{-1}$
\eqref{int:to:bound232:+} converges.  If necessary, we then increase the size of $B_-$ so that $A<\varepsilon+B_-[F;\QQ]^{-1}$; this will make \eqref{int:to:bound232:-} convergent.  We deduce that we can take
$$
\beta_{n,F}:=\max(\varepsilon+B_+[F:\QQ]^{-1},\varepsilon+A')\,.
$$
\end{proof}
 
\begin{lem} \label{lem:Eis:bound} Let $A>\beta_{F,n}$ where $\beta_{F,n}$ is the constant of Lemma \ref{lem:Eis:bound0}.  Then for any $N \in \ZZ_{ \geq 0}$ one has
$$
\mathrm{sup}_{\mathrm{Re}(s)=A} |E(g,\Phi_{\chi_s}) |\ll_{\Phi,A,N} C(\chi_s)^{-N}\,.
$$
\end{lem}

\begin{proof}
We proceed as in the Riemann-Lebesgue lemma, leveraging the smoothness of $\Phi_v$
for $v|\infty$ the form of Lemma \ref{lem:DO} to obtain bounds on the Mellin transform $\Phi_{\chi_s}$ and hence $E(g,\Phi_{\chi_s})$.

For every $v|\infty$ let $D_v$ be defined as in \eqref{D} and if $v$ is complex let $\bar{D}_v$ be defined as in \eqref{barD}.  If $\chi_v=|\cdot|\mu^\alpha$ as in \eqref{explicit:chi} then using the same computation as in the proof of Lemma \ref{lem:DO} we have
$$
\left|\left(\frac{([F:\RR]-1)\alpha}{2}+it+s+\tfrac{n+1}{2}\right)E(g,\Phi_{\chi_s})\right|_{\mathrm{st}}=\left|E(g,(D_v\Phi)_{\chi_s}) \right|_{\mathrm{st}} \ll_{D_v\Phi,A} 1\,.
$$
If we replace $\alpha$ by $-\alpha$ and $D_v$ by $\bar{D}_v$ when $v$ is complex the same inequality and bound are valid.  In view of the definition \eqref{AC:def} of the analytic conductor we deduce the lemma by induction.
\end{proof}

\begin{proof}[Proof of Theorem \ref{thm:Eis:bound}]
The main idea here is to use the Phragmen-Lindel\"of principle.    This is complicated by the same difficulties as those overcome in \cite{Gelbart:Shahidi:JAMS}. We adapt the simplification of their argument given in \cite{Gelbart:Lapid}.

Let $V \subset \CC$ be a simply connected open subset.  An entire function $f: V \to \CC$ is said to be of finite order $A \geq 0$ if 
$$
|f(z)| \ll e^{c|z|^A}
$$
for some $c \in \RR_{>0}$.  
By \cite[Theorem 0.2]{Mueller:Residual} for each fixed $g$ the function $E(g,\Phi_{\chi_s})$ is the quotient of two functions of finite order.  

By Theorem \ref{thm:Ikeda} there is a polynomial $P_0 \in \CC[x]$ satisfying $P_0(-x)=P_0(x)$, independent of 
$\chi$, $\Phi$, and $g$, such that 
$$
P_0(s)E(g,\Phi_{\chi_s})
$$
is holomorphic as a function of $s$.  It is therefore a function of finite order by \cite[\S 2.3 Lemma 1]{Gelbart:Lapid}.
For real numbers $A<B$ let
$$
V_{A,B}:=\{s:A \leq  \mathrm{Re}(s) \leq B\}.
$$ 
We take $A<-\beta_{F,n}$ and $B>\beta_{F,n}$.  Then by Lemma \ref{lem:Eis:bound}
\begin{align} \label{bound:A}
|E(g,\Phi_{\chi_s})|_{\mathrm{st}} &\ll_N C(\chi_s)^{-N} \,,
\end{align}
for $\mathrm{Re}(s)=B$.
By the functional equation of Theorem \ref{thm:Ikeda} one hs $E(g,\Phi_{\chi_s})=E(g,M^*_{w_0}(\Phi_{\chi_s}))$.  Moreover by Theorem \ref{thm:FT} one has $E(g,M^*_{w_0}(\Phi_{\chi_s}))=E(g,\mathcal{F}(\Phi)_{\bar{\chi}_{-s}})$ for $\mathrm{Re}(s)=A$.  Thus we can apply Lemma \ref{lem:Eis:bound} to $E(g,\mathcal{F}(\Phi)_{\bar{\chi}_{-s}})$ to deduce
\begin{align} \label{bound:B}
|E(g,\Phi_{\chi_s}))|_{\mathrm{st}}=| E(g,M_{w_0}^*(\Phi_{\chi_s}))|_{\mathrm{st}}=|E(g,\mathcal{F}(\Phi)_{\bar{\chi}_{-s}})|_{\mathrm{st}} &\ll C(\bar{\chi}_{-s})^{-N}\,,
\end{align}
for $\mathrm{Re}(s)=A$ (note that $A<-\beta_{F,n}$).
Since $\Phi$ is $K$-finite for any $\chi$ with  $\Phi_{\chi_s} \neq 0$ we have $C(\bar{\chi}_{-s}) \asymp C(\chi_s)$ on the line $\mathrm{Re}(s)=A$.

Applying the Phragmen-Lindel\"of principle (in the form of \cite[\S III.4, Theorem 11]{Moreno}, for example) the bounds \eqref{bound:A} and \eqref{bound:B} on the edge of the vertical strip imply analogous bounds on its interior and we deduce the theorem.
\end{proof}

Abbreviate $E(\Phi_{\chi_{s}})=E(I_{2n},\Phi_{\chi_{s}})$.  Let 
$$
\kappa_F:=\mathrm{Res}_{s=1}\zeta_F(s)\,.
$$
The main theorem of this paper is the following:
\begin{thm} \label{thm:main}
Let $\Phi \in \mathcal{S}(X(\A_F),K)$.  One has
\begin{align*}
&\sum_{\gamma \in X(F)}\Phi(\gamma)+
\frac{1}{\kappa_F}\sum_{\substack{
0 \leq m <\frac{n+1}{2} \\m \in \mathbb{Z}}}\mathrm{Res}_{s= \frac{n+1}{2}-m }E(\mathcal{F}(\Phi)_{1_s})+\frac{1}{\kappa_F}\sum_{\substack{\chi \in \widehat{[\GG_m]} \\\chi \neq 1, \chi^2=1}}
\sum_{\substack{
0 \leq m <\frac{n-1}{2} \\m \in \mathbb{Z}}}
\mathrm{Res}_{s=\frac{n-1}{2}-m }E(\mathcal{F}(\Phi)_{\chi_s})
\\
&=\sum_{\gamma \in X(F)} \mathcal{F}(\Phi)(\gamma)+
\frac{1}{\kappa_F}\sum_{\substack{
0 \leq m <\frac{n+1}{2} \\m \in \mathbb{Z}}}\mathrm{Res}_{s= \frac{n+1}{2}-m }E(\Phi_{1_s})+\frac{1}{\kappa_F}\sum_{\substack{\chi \in \widehat{[\GG_m]} \\\chi \neq 1, \chi^2=1}}
\sum_{\substack{0 \leq m <\frac{n-1}{2} \\m \in \mathbb{Z}}}
\mathrm{Res}_{s=\frac{n-1}{2}-m }E(\Phi_{\chi_s})\,.
\end{align*}
All of the sums here are absolutely convergent.  
\end{thm}

\begin{proof}

For $m \in M(\A_F)$ the adelic Mellin transform of $\sum_{\gamma \in X(F)} \Phi(m^{-1}\gamma)$ evaluated at $\chi_s$ is $E(\Phi_{\chi_s})$:
\begin{align}
E(\Phi_{\chi_s})=\int_{M^{\mathrm{ab}}(F) \backslash M^{\mathrm{ab}}(\A_F)}\sum_{\gamma \in X(F)}\chi_s(m)\delta_P(m)^{1/2}\Phi(m^{-1}\gamma)dm\,.
\end{align}
By the argument proving Lemma \ref{lem:Eis:bound0}, the integral and sum are absolutely convergent for $\mathrm{Re}(s)>\beta_{F,n}$.

Applying Poisson summation in $F^\times$ we see that
\begin{align*}
\sum_{\gamma \in X(F)} \Phi(\gamma)
&=\sum_{\chi}\frac{1}{2\pi i\kappa_F}
\int_{\mathrm{Re}(s)=\sigma}E(\Phi_{\chi_{s}})ds\\
&=\sum_{\chi}\frac{1}{2\pi i\kappa_F}
\int_{\mathrm{Re}(s)=\sigma}E(M_{w_0}^*(\Phi_{\chi_{s}}))ds\,,
\end{align*}
for $\sigma>\beta_{F,n}$.  Here the sum on $\chi$ is over $\widehat{[\GG_m]}$.  
A convenient reference for this application of Poisson summation is \cite[\S 2]{Blomer_Brumley_Ramanujan_Annals}.  In view of Lemma \ref{lem:sum:bound} to justify the application it suffices to check that 
\begin{align} \label{to:bound0}
\sum_{\chi}
\int_{\mathrm{Re}(s)=\sigma}|E(M_{w_0}^*(\Phi_{\chi_{s}}))|ds
\end{align}
is finite.
Let $K_{\GG_m} \leq \A_F^\times$ be the maximal compact subgroup.  There is a finite set of $K_{\GG_m}$-types such that all characters contributing a nonzero summand to \eqref{to:bound0} have $K_{\GG_m}$-type in that set.  
On the other hand, one can readily check that for large enough $A>0$
$$
\sum_{\chi} \int_{\mathrm{Re}(s)=\sigma} C(\chi_s)^{-A}<\infty\,,
$$
where the sum is over all characters $\chi$ whose $K_{\GG_m}$ type lies in a fixed finite set of $K_{\GG_m}$-types.  With this in mind Theorem \ref{thm:Eis:bound} implies that \eqref{to:bound0} is finite.

We now shift the $s$ contour to $-\sigma$.  We arrive at the sum of
\begin{align} \label{no:residues0}
\sum_{\chi}\frac{1}{2\pi i\kappa_F}
\int_{\mathrm{Re}(s)=-\sigma}E(M_{w_0}^*(\Phi_{\chi_{s}}))ds
\end{align}
and the contribution of the residues:
\begin{align} \label{residues}
&\frac{1}{\kappa_F}\sum_{
\substack{0 \leq m <\frac{n+1}{2} \\ m\in \mathbb{Z}}}\mathrm{Res}_{s=\pm \left(\frac{n+1}{2}-m \right)}E(M^*_{w_0}(\Phi_{1_{s}}))\\&+\frac{1}{\kappa_F}\sum_{\substack{\chi \in \widehat{[\GG_m]} \\ \chi \neq 1, \chi^2=1}}
\sum_{\substack{0 \leq m <\frac{n-1}{2} \\ m\in \mathbb{Z}}}
\mathrm{Res}_{s=\pm \left( \frac{n-1}{2}-m \right)}E(M^*_{w_0}(\Phi_{\chi_{s}}))\,. \nonumber
\end{align}
The characters contributing a nonzero summand are all quadratic or trivial and they all have conductor dividing an ideal depending only on $\Phi$.  Thus the sums in \eqref{residues} are in fact finite.

In \eqref{no:residues0} we change variables $\chi \mapsto \bar{\chi}$ and $s \mapsto -s$ to arrive at 
\begin{align} \label{no:residues}
\sum_{\chi}\frac{1}{2\pi i\kappa_F}
\int_{\mathrm{Re}(s)=\sigma}E(M_{w_0}^*(\Phi_{\bar{\chi}_{-s}}))ds\,.
\end{align}

Reversing the application of Poisson summation we see that \eqref{no:residues} is equal to 
\begin{align}
\sum_{\gamma \in X(F)} \mathcal{F}(\Phi)(\gamma)\,.  
\end{align}
This is again absolutely convergent by Lemma \ref{lem:sum:bound}.

To complete the proof we now write the contribution of the residues \eqref{residues} in the more symmetric form stated in the theorem using the fact that $\mathcal{F}(\Phi)_{\chi_s}=M_{w_0}^*(\Phi_{\bar{\chi}_{-s}})$ and Theorem \ref{thm:Ikeda}.
\end{proof}

\appendix

\section{Proof of Proposition \ref{prop:compact} in the Archimedean case} \label{App}

Assume that $F$ is an Archimedean local field.  
Recall that $C_c^\infty(X(F),K) < C_c^\infty(X(F))$
is the subset of functions that are right $K$-finite.  
Our goal here is to prove Proposition \ref{prop:compact} in the current Archimedean setting.  We recall that Proposition \ref{prop:compact} simply states that $C_c^\infty(X(F),K) \leq \mathcal{S}(X(F),K)$. 

It is convenient to begin with a few convergence lemmas:
\begin{lem}  \label{lem:abs:conv1} Let $\Phi \in C_c^\infty(X(F),K)$.  Then for $\mathrm{Re}(s) \geq n^2/2$ one has
$$
|M_{w_0}\Phi_{\chi_s}(I_{2n})| \ll_{\Phi} 1\,.
$$
\end{lem}

\begin{proof}

Consider the map 
\begin{align} \label{Pl:copy}
\mathrm{Pl}:X(F) \lto \wedge^n F^{2n}-0\,.
\end{align}
Since $X$ is a homogeneous space for $\mathrm{Sp}_{2n}$ it is smooth (as a scheme over $F$).  The map \eqref{Pl:copy} is an injective diffeomorphism onto its image, a closed submanifold of $\wedge^n F^{2n}-0$.
In particular, there is a function $\Psi \in C_c^\infty(\wedge^n F^{2n})$ such that $\Phi=\Psi \circ \mathrm{Pl}$.
For $\mathrm{Re}(s)$ sufficiently large we have 
\begin{align*}
M_{w_0} \Phi_{\chi_s}(I_{2n})&=\int_{N(F)} \Phi_{\chi_s}(w_0^{-1}n)dn\\
&=\int_{N(F)}\left(\int_{M^{\mathrm{ab}}(F)}\delta_P^{1/2}(m)\chi_s(\omega(m))\Psi( \mathrm{Pl}(m^{-1}w_0^{-1}n))dm^\times \right)dn\,.
\end{align*}
Temporarily denote by 
\begin{align*}
\mathrm{Pl}_0:M_{n \times n}^{\oplus 2}(F) &\lto \wedge^{n}F^{2n}\\
(X,Y) &\longmapsto \mathrm{Pl}\left( \begin{smallmatrix} * & *\\ X & Y \end{smallmatrix}\right)\,.
\end{align*}
This is just taking the wedge product of the $n$ rows of the $n \times 2n$ matrix $(X\, Y)$, going from top to bottom.  
Then the integral above can be written
\begin{align*}
&\int_{\mathrm{Sym}^n(F)}\left(\int_{F^\times}\chi_s(a)|a|^{(n+1)/2}\Psi( \mathrm{Pl}_0\left(-\left(\begin{smallmatrix} a & \\ & I_{n-1}\end{smallmatrix} \right)J',-\left(\begin{smallmatrix} a & \\ & I_{n-1}\end{smallmatrix} \right)J'z \right))da^\times \right)dz\,,
\end{align*}
where $\mathrm{Sym}^n(F)$ is the $F$-vector space of symmetric $n \times n$ matrices and 
$$
J'=\left(\begin{smallmatrix} & & 1 \\ & \reflectbox{$\ddots$} & \\ 1 & &  \end{smallmatrix}\right).
$$
We note that $\mathrm{Pl}_0$ is invariant under multiplication by $\mathrm{SL}_n$ on the left to see that the above is equal to 
\begin{align*}
&\int_{\mathrm{Sym}^n(F)}\left(\int_{F^\times}\chi_s(a)|a|^{(n+1)/2}\Psi( -\mathrm{Pl}_0\left(\left(\begin{smallmatrix} I_{n-1}& \\ & a\end{smallmatrix} \right),\left(\begin{smallmatrix} I_{n-1}& \\ & a\end{smallmatrix} \right)z \right))da^\times \right)dz\,.
\end{align*}
Take a change of variables $z \mapsto a^{-1}z$ to arrive at
\begin{align*}
&\int_{\mathrm{Sym}^n(F)}\left(\int_{F^\times}\chi_{s+(1-n^2)/2}(a)\Psi(- \mathrm{Pl}_0\left(\left(\begin{smallmatrix} I_{n-1} & \\ & a\end{smallmatrix} \right),\left(\begin{smallmatrix} I_{n-1} & \\ & a\end{smallmatrix} \right)a^{-1}z \right))da^\times \right)dz\,.
\end{align*}
By inspection this is rapidly decreasing as a function of $z \in \mathrm{Sym}^n(F)$ and $a \in F$, so this integral converges absolutely and is bounded by a constant depending only on $\Psi$ for $\mathrm{Re}(s)\geq n^2/2$. 
\end{proof}

\begin{lem} \label{lem:abs:conv2}
 Let $\Phi \in C_c^\infty(X(F),K)$ and $A \in \RR$.  Then for $\mathrm{Re}(s) \geq A$ one has
 \begin{align*}
|M_{w_0}\Phi_{\chi_s}(I_{2n})| \ll_{\Phi,A} 1\,.
 \end{align*}
\end{lem}

\begin{proof}
Let $\alpha \in \ZZ_{\geq 0}$.  Notice that 
$$
\widetilde{\Phi}:=(\omega\bar{\omega})^{\alpha}\Phi \in C_c^\infty(X(F),K)\,,
$$
where $\omega\bar{\omega}$ is defined as in \eqref{omega}.  As in the proof of Lemma \ref{lem:char} we have
\begin{align*}
((\omega \bar{\omega})^{-\alpha}\widetilde{\Phi})_{\chi_s}
&=\widetilde{\Phi}_{\chi_{s+2\alpha[F:\RR]^{-1}}}\,.
\end{align*}
By Lemma \ref{lem:abs:conv1} we therefore have have
$$
|M_{w_0}\Phi_{\chi_s}(I_{2n})|=
|M_{w_0}((\omega \bar{\omega})^{-\alpha}\widetilde{\Phi})_{\chi_s}(I_{2n})|=|M_{w_0}
\widetilde{\Phi}_{\chi_{s+2\alpha[F:\RR]^{-1}}}(I_{2n})| \ll_{\Phi,\alpha} 1\,,
$$
for 
$$
\mathrm{Re}(s)+2\alpha [F:\RR]^{-1}\geq n^2/2\,.
$$  Taking $\alpha$ sufficiently large we deduce the lemma.  
\end{proof}

\begin{proof}[Proof of Proposition \ref{prop:compact} in the Archimedean case]
If $\Phi \in C_c^\infty(X(F),K)$ then it is easy to see that 
$\Phi_{\chi_s}$ is holomorphic for all $\chi$ and hence by \cite[Lemma 1.3]{Ikeda:poles:triple} we deduce that $\Phi_{\chi_s}$ is a good section.  We thus have to verify that for all $g \in \mathrm{Sp}_{2n}(F)$, all characters $\chi$, $A<B$, and all $P_w$ as in the definition of an excellent section that 
\begin{align} \label{point1}
|\Phi_{\chi_s}(g)|_{A,B,P_{\mathrm{Id}}}
\end{align}
and
\begin{align} \label{point2}
|M_{w_0}\Phi_{\chi_s}(g)|_{A,B,P_{w_0}}
\end{align} 
are finite.  In fact this is enough to complete the proof since the space $C_c^\infty(X(F),K)$ is preserved under the differential operators $D$ and $\bar{D}$ of \eqref{D}.

Write $\chi$ as in \eqref{explicit:chi}.
By Lemma \ref{lem:DO} for any $N,N' \in \ZZ_{ \geq 0}$ with $N'=0$, one has
\begin{align*}
|D^N\Phi_{\chi_s}|_{A,B,1}=\left|\left(it+s+\frac{n+1}{2} \right)^N\Phi_{\chi_s}(g)\right|_{A,B,1}\,,
\end{align*}
if $F$ is real, and 
\begin{align*}
|D^N\bar{D}^{N'}\Phi_{\chi_s}(g)|_{A,B,1}=
\left|\left(\frac{\alpha}{2}+it+s+\frac{n+1}{2} \right)^N\left(-\frac{\alpha}{2}+it+s+\frac{n+1}{2} \right)^{N'}\Phi_{\chi_s}(g)\right|_{A,B,1}\,,
\end{align*}
if $F$ is complex.
Since $\Phi \in C_c^\infty(X(F),K)$ the left hand sides here are bounded by a constant depending on $A,B$ $\Phi,N,N'$.  We deduce that \eqref{point1} is finite for all $P \in \CC[x]$.  An analogous argument, using Lemma \ref{lem:abs:conv2}, allows us to deduce that \eqref{point2} is finite for all $P \in \CC[x]$.
\end{proof}


\bibliography{refs}
\bibliographystyle{alpha}

\end{document}